\documentclass{amsart}
\usepackage{amssymb}
\usepackage{braket}
\usepackage{mathrsfs}
\usepackage{ifthen}
\usepackage{todonotes}

\usepackage{tikz}
\usetikzlibrary{patterns}
\usepackage{comment} 
\usepackage{mleftright}
\usepackage{hyperref}
\usepackage{cleveref}
 \usepackage[all]{xy}
 \usepackage{amscd}
 \usepackage{amsrefs}
 \usepackage{color}
 \usepackage{enumitem}
\newlist{steps}{enumerate}{1}
\setlist[steps, 1]{label = Step \arabic*:}
\usepackage[abs]{overpic}
\usepackage[latin1]{inputenc}

\makeatletter
\DeclareRobustCommand\widecheck[1]{{\mathpalette\@widecheck{#1}}}
\def\@widecheck#1#2{%
   \setbox\z@\hbox{\m@th$#1#2$}%
   \setbox\tw@\hbox{\m@th$#1%
      \widehat{%
         \vrule\@width\z@\@height\ht\z@
         \vrule\@height\z@\@width\wd\z@}$}%
   \dp\tw@-\ht\z@
   \@tempdima\ht\z@ \advance\@tempdima2\ht\tw@ \divide\@tempdima\thr@@
   \setbox\tw@\hbox{%
      \raise\@tempdima\hbox{\scalebox{1}[-1]{\lower\@tempdima\box\tw@}}}%
   {\ooalign{\box\tw@ \cr \box\z@}}}
\makeatother

\theoremstyle{plain}
\newtheorem{thm}{Theorem}[section]
\crefname{thm}{Theorem}{Theorems}
\Crefname{thm}{Theorem}{Theorems}
\newtheorem{prop}[thm]{Proposition}
\crefname{prop}{Proposition}{Propositions}
\Crefname{prop}{Proposition}{Propositions}
\newtheorem{lem}[thm]{Lemma}
\crefname{lem}{Lemma}{Lemmas}
\Crefname{lem}{Lemma}{Lemmas}
\newtheorem{cor}[thm]{Corollary}
\crefname{cor}{Corollary}{Corollaries}
\Crefname{cor}{Corollary}{Corollaries}

\crefname{claim}{Claim}{Claims}
\Crefname{claim}{Claim}{Claims}

\crefname{property}{Property}{Properties}
\Crefname{property}{Property}{Properties}

\crefname{problem}{Problem}{Problems}
\Crefname{problem}{Problem}{Problems}
\newtheorem{conjecture}[thm]{Conjecture}
\crefname{conjecture}{Conjecture}{Conjecture}
\Crefname{conjecture}{Conjecture}{Conjecture}

\theoremstyle{definition}
\newtheorem{defn}[thm]{Definition}
\crefname{defn}{Definition}{Definitions}
\Crefname{defn}{Definition}{Definitions}

\crefname{notation}{Notation}{Notations}
\Crefname{notation}{Notation}{Notations}

\crefname{convention}{Convention}{Conventions}
\Crefname{convention}{Convention}{Conventions}

\crefname{cond}{Condition}{Conditions}
\Crefname{cond}{Condition}{Conditions}

\crefname{assum}{Assumption}{Assumptions}
\Crefname{assum}{Assumption}{Assumptions}

\crefname{conj}{Conjecture}{Conjectures}
\Crefname{conj}{Conjecture}{Conjectures}

\theoremstyle{remark}
\newtheorem{rem}[thm]{Remark}
\crefname{rem}{Remark}{Remarks}
\Crefname{rem}{Remark}{Remarks}

\crefname{ex}{Example}{Examples}
\Crefname{ex}{Example}{Examples}

\crefname{section}{Section}{Sections}
\Crefname{section}{Section}{Sections}
\crefname{subsection}{Subsection}{Subsections}
\Crefname{subsection}{Subsection}{Subsections}
\crefname{figure}{Figure}{Figures}
\Crefname{figure}{Figure}{Figures}

\newtheorem*{acknowledgement}{Acknowledgement}

\newcommand{\Z}{\mathbb{Z}}

\newcommand{\Q}{\mathbb{Q}}

\newcommand{\Coker}{\mathop{\mathrm{Coker}}\nolimits}
\newcommand{\im}{\operatorname{Im}}

\newcommand{\Hom}{\mathop{\mathrm{Hom}}\nolimits}

\newcommand{\supp}{\mathop{\mathrm{supp}}\nolimits}

\newcommand{\id}{\mathrm{id}}
\newcommand{\ind}{\mathop{\mathrm{ind}}\nolimits}

\newcommand{\G}{\mathcal G}

\newcommand{\C}{\mathbb{C}}
\newcommand{\s}{\mathfrak{s}}
\newcommand{\wh}{\widehat}

\newcommand{\pr}{\text{pr}}
\newcommand{\al}{\alpha}

\def\om{\omega}
\def\Om{\Omega}

\def\Spinc{\text{Spin}^c}

\newcommand{\R}{\mathbb R}

\def\ker{\operatorname{Ker}}

\def\dim{\operatorname{dim}}

\def\Hom{\operatorname{Hom}}

\def\Con{\mathcal{C}}

\def\id{\operatorname{Id}}

\def\ind{\operatorname{ind}}

\def\cL{\mathcal{L}}
\newcommand{\mbar}[1]{{\ooalign{\hfil#1\hfil\crcr\raise.167ex\hbox{--}}}}

\def\dvol{d\operatorname{vol}}

\def\cV{\mathcal{V}}

\def\cF{\mathcal{F}}

\def\cU{\mathcal{U}}

\def\wt{\widetilde}

\title[Adjunction inequality, BF-type  refinement  of  KM-invariant]{An adjunction inequality for the Bauer--Furuta type invariants, with applications to sliceness and 4-manifold topology
}

\author{Nobuo Iida}
\address{Graduate School of Mathematical Sciences, the University of Tokyo, 3-8-1 Komaba, Meguro, Tokyo 153-8914, Japan}
\email{iida@ms.u-tokyo.ac.jp}

\author{Anubhav Mukherjee}

\address{School of Mathematics \\ Georgia Institute
of Technology \\  Atlanta  \\ Georgia \\ USA}

\email{anubhavmaths@gatech.edu}

\author{Masaki Taniguchi}

\address{2-1 Hirosawa, Wako, Saitama 351-0198, Japan}
\email{masaki.taniguchi@riken.jp}

\date{}
\begin{document}

\begin{abstract} 
Our main result gives an adjunction inequality for embedded surfaces in certain $4$-manifolds with contact boundary under a non-vanishing assumption on the Bauer--Furuta type invariants. Using this, we give infinitely many knots in $S^3$ that are not smoothly H-slice (that is, bounding a null-homologous disk) in many $4$-manifolds but they are topologically H-slice. In particular, we give such knots in the boundaries of the punctured elliptic surfaces $E(2n)$. In addition, we give obstructions to codimension-0 orientation-reversing embedding of weak symplectic fillings with $b_3=0$ into closed symplectic 4-manifolds with $b_1=0$ and $b_2^+\equiv 3 \operatorname{mod} 4$. From here we prove a Bennequin type inequality for symplectic caps of $(S^3,\xi_{std})$. We also show that any weakly symplectically fillable $3$-manifold bounds a $4$-manifold with at least two smooth structures.

\end{abstract}

\maketitle

\section{Introduction} 


The adjunction inequality is one of the important topics in low dimensional topology that put constraints on the genus of embedded surfaces in 4-manifolds with non-trivial gauge-theoretic invariants. Kronheimer and Mrowka have proven an adjunction-type inequality for 4-manifolds with non-trivial Donaldson polynomial invariants in \cite{KM93, KM95, KM95II} and used to give an affirmative answer to the long standing Milnor conjecture about 4-genus of torus knots. 
After Seiberg--Witten theory came into the picture, the adjunction inequality was proven for $\C \mathbb{P}^2$ by Kronheimer and Mrowka \cite{KM94} and later generalized to the other manifolds. (For generalizations, see \cite{HT16, Ko16, St03, Fr05, FKM01, MST96, MR06, OS00, OS06, KLS18II}.)
Moreover, there are several relative versions (i.e. for surfaces with boundary in 4-manifolds with boundary) of adjunction inequalities \cite{MR06, HR20, manolescu}.  
These inequalities have many applications such as, detecting smooth 4-genus, finding exotic pairs of 4-manifolds and obstructing symplectic structures. 
 Recently, a relation between exotic smooth structures and (H-)sliceness attracts attention \cite{manolescu, MP21}. 

 Our main theorems which give applications in Subsection 1.2, 1.3 and 1.4 are adjunction-type inequalities for the following gauge theoretic invariants: 
 a Bauer--Furuta type refinement $
    \Psi(W, \xi , \s )$ \cite{I19} of the Kronheimer--Mrowka's invariant \cite{KM97} for a compact spin$^c$ 4-manifold $W$ with contact boundary $(Y,\xi)$ and $b_3(W)=0$, which is defined by the first author.

More specifically, we have the following result which, for simplicity, we stated only for the genus-0 case. See \cref{main2'} for a similar inequality for higher genus surfaces.

\begin{thm}\label{main2} 
Let $(W, \s)$ be an compact, oriented, Spin$^c$ 4-manifold with $b_3(W)=0$ whose boundary is a contact $3$-manifold $(Y, \xi)$  and $\s|_Y = \s_\xi$ ( here $\s_\xi$ denotes the induced Spin$^c$-structure from the 2-plane field $ \xi$).
If $\Psi (W, \xi, \s) \neq 0$, then non-torsion homology class in $H_2(W, \partial W; \Z)$ cannot be realized by a smoothly embedded 2-sphere whose self-intersection number is non-negative.
\end{thm}
\begin{rem}
In the proof of \cref{main2}, we adopt a {\it neck-stretching argument} developed in \cite{Fr05} and \cite{KM94} to obtain an adjunction type inequality in Seiberg--Witten theory. 
A difficulity to prove \cref{main2} comes from non-compactness of base 4-manifolds appearing in the definition of the Bauer--Furuta type refinement $\Psi(W, \xi , \s )$ \cite{I19} of Kronheimer--Mrowka's invariant.   
\end{rem}

For our applications in Subsection 1.2-1.4, we combine \cref{main2} with the connected sum formula \cite[Theorem 4.4]{I19} and the non-vanishing result \cite[Corollary 4.3]{I19} of the first author's invariant.
Also, we need to establish the following non-vanishing result: 
\begin{thm}\label{b=1'}
Let $(W, \om)$ be a weak symplectic filling of a contact 3-manifold $(Y,\xi)$ with $b_3(W)=0$. Suppose $Y$ is a rational homology 3-sphere with positive scalar curvature metric. 
We also consider a closed $\Spinc$ 4-manifold $(X, \s_X)$ with $b_1(X)=0$ such that the $S^1$-equivariant Bauer--Furuta invariant $BF(X, \s_X)$ is not $S^1$-null-homotopic. 
Then,  the relative Bauer--Furuta invariant $BF (W\# X, \s_{\om} \# \s_X)$ is equal to $BF(X, \s_X)$ and in particular does not vanish as an $S^1$-equivariant stable homotopy class.
 
\end{thm}
The proof of \cref{b=1'} uses a Floer homotopy contact invariant introduced in \cite[Theorem 1.1]{IT20} and a gluing result proven in \cite[Theorem 1.2]{IT20}. 

We now list some applications:

\subsection{Bennequin type inequality for symplectic caps and its application to obstruct sliceness of knots.} 

It is a classical problem to understand the topology of symplectic 4-manifolds with contact boundary. There are two types of such manifolds: one is the case when the boundary contact structure is convex with respect to the symplectic form, we call them {\it symplectic fillings}, and the other is the case when the boundary contact structure is concave, known as {\it symplectic caps}. (See Section~\ref{symplectic} for details.) The smooth topology of symplectic fillings is very rigid. For example, Gromov proved that $S^3$ has a unique symplectically fillable contact structure $\xi_{std}$ and every symplectic filling is diffeomorphic to the 4-ball up to blow-ups \cite{Gromov1985}. Moreover, many contact 3-manifolds have no symplectic fillings. On the other hand, every contact 3-manifold has a symplectic cap \cite{ehcap} and one can easily construct a symplectic cap of $(S^3,\xi_{std})$ by starting with a closed symplectic 4-manifold and removing a Darboux 4-ball. Thus, the topology of symplectic caps are harder to study in general and much less is known. 

Historically, there are many tools from gauge theory which apply well to symplectic fillings of contact 3-manifolds. These tools allow one to demonstrate, for example, that various knots are not slice in the filling or that some fillings can not be embedded into some  4-manifolds. Mrowka--Rollin's generalized Bennequin inequality \cite[Theorem A]{MR06} can be useful to understand the minimal genus of embedded smooth surfaces in a given filling, bounded by a fixed knot. We provides the first examples of results of this type that can be applied to symplectic caps rather than fillings. Since symplectic caps are far more ubiquitous, this is a dramatic improvement. In particular, the fact that $S^3$ admits non-trivial symplectic caps means that these tools can approach some important problems on relative genus bounds in symplectic caps. 

\begin{thm}\label{ben}
Let $(X,\omega)$ be a symplectic cap for $(S^3,\xi_{std})$ with $b_1(X)= 0$ and $b^+_2(X) \equiv 3 \operatorname{mod} 4$ and $K$ be a knot in $S^3$ If $K$ has a Legendrian representative in $(S^3,\xi_{std})$ with $tb(K)>0$, then the knot $K$ does not bound a smooth disk $D$ in $X$ with $[D]\cdot [D]\geq 0$. Here $tb(K)\in \Z$ is the Thurston--Bennequin number of $K$ defined in Section~\ref{contact}.

\end{thm}
Note that when we talk about smooth 4-manifold with boundary, we usually consider the positive orientation on the boundary, whereas when we consider a symplectic cap, the boundary is negatively oriented.

\begin{rem}
 In Theorem~\ref{ben}, the assumption $b^+_2(X) \equiv 3 \operatorname{mod} 4$ is important. 
  For symplectic caps with $b^+_2(X) \equiv 1 \operatorname{mod} 4$, the above Bennequin type inequality does not hold, see Remark~\ref{remark}. Also notice that, in the case of  $b^+_2(X) \equiv 3 \operatorname{mod} 4$, the knot $K$ still can bound a disk of negative self-intersection number, see e.g. Theorem~\ref{top h}. This discussion also demonstrates how difficult it is to control the topology of symplectic caps of a contact 3-manifold as opposed to symplectic fillings. 
\end{rem}
\begin{rem}
Although \cref{ben} obstructs the existence of embedded smooth disks, we can prove a higher genus version by a similar technique. For the precise statement, see \cref{additional}. 
\end{rem}

Theorem~\ref{ben} can give obstructions to the existence of H-slice disks in general 4-manifolds with $S^3$-boundary.
 First, let us explain the definition of H-slice. Let $X$ be a closed, oriented, connected, smooth 4-manifold and let $K$ be a knot in $S^3$. 
In \cite[Definition 6.2]{MMSW19}, Manolescu, Marengon, Sarkar, and Willis defined $K$ to be \textit{smoothly (resp. topologically) H-slice} in $X$ if $K$ bounds a properly embedded smooth (resp. locally-flat)  null-homologous disc in $X^{\circ} = X- \operatorname{Int}{B^4}$.
 There are several known obstructions to smooth H-sliceness in both definite 4-manifolds \cite{OS03, KM13, MMSW19} and indefinite 4-manifolds \cite{manolescu}. Our technique is the first one where the contact structure of the boundary $S^3$ has been used to find obstructions to H-sliceness.  

 Since $S^3$ has a unique tight contact structure \cite{utight}, we denoted by $TB(K)$ the maximal value of Thurston--Bennequin numbers as among all Legendrian representatives of $K$ in $(S^3 ,\xi_{std})$ (for more details, check Section~\ref{contact}). As a corollary of \cref{ben}, we have the following obstruction to H-sliceness: 
\begin{cor}\label{h-slice}
Let $X$ be a closed, symplectic $4$-manifold with $b_1(X)=0$ and $b_2^+(X)\equiv 3 \operatorname{mod} 4$. If $K$ is a knot in $S^3$ with $TB(K)>0$, then it is not smoothly $H$-slice in $X$. 
\end{cor}

Using our techniques, we can actually produce a very large class of exotically H-slice knots, i.e. knots that are topologically but not smoothly H-slice in various 4-manifolds.

\begin{cor}\label{infinite}
There is an infinite family of knots $K_i, i\in \Z$ in $S^3$ that are linearly independent in the knot concordance group such that are topologically $H$-slice but not smoothly $H$-slice in any closed symplectic $4$-manifold $X$ with $b_1(X)=0$ and $b_2^+(X)\equiv 3 \operatorname{mod} 4$.
\end{cor}
\begin{rem}\label{csum}
 We will also show that Corollaries~\ref{h-slice} and ~\ref{infinite} are true when $X$ is replaced with $X_1,  X_1\#X_2,$ or $X_1\#X_2\#X_3$, where $X_i, i=1,2,3,$ are manifolds satisfying the hypothesis of $X$ in the above theorems. We also can take $K_i$ to be topologically $H$-slice knots in $S^4$.
\end{rem}


When $X$ is $K3$, we can do even better. 
\begin{thm}\label{top h}
Let $X$ be $K3, K3\# K3$, or $K3\#K3\#K3$. There is an infinite family of knots $K_i, i\in \Z$ in $S^3$ that are 
\begin{itemize} 

\item linearly independent in the smooth knot concordance group, \item topologically $H$-slice in $S^4$, and in particular in $X$, 
\item smoothly slice in $X$, but 
\item not smoothly $H$-slice in $X$.
\end{itemize}
\end{thm}
\begin{rem}

Note that, in \cite[Corollary 1.7]{manolescu}, it is proved that there is a topologically H-slice knot in $K3$ but not smoothly H-slice in $K3$ by using the {\it $10/8$-theorem} \cite{manolescu, Fu01, HLSX18}. 
This is the first linearly independent infinite sequence of such examples.
\end{rem}

Our techniques can also be useful to obstruct sliceness of a given knot in a rational homology 4-ball $X$ with $\partial X = S^3$. 
Similar results have been proven before using Heegaard Floer theory \cite{P04}.

\begin{thm}\label{slice}
Let $K$ be a knot in $S^3$ with either $TB(K)>0$ or $TB(\overline{K})>0$, where $\overline{K}$ is the mirror image of $K$.  Then $K$ is not slice in any rational homology ball $W$ with boundary $S^3$.
\end{thm}
\begin{cor}\label{neg def inf}
There is an infinite family of knots $K_i, i\in \Z$ in $S^3$ that are linearly independent in the knot concordance group and that are topologically $H$-slice but not smoothly $H$-slice in any closed negative definite $4$-manifold $X$ with $b_1(X)=0$.
\end{cor}
A stronger result is also proved in \cite{Sa18} for both positive and negative 4-manifolds. 

\subsection{Constraints on the topology of symplectic caps}
For a simply-connected closed 4-manifold, it is asked in \cite[Problem 4.18]{K78} whether there is a handle decomposition without 1- and 3-handles. A 4-manifold having such a handle decomposition is called a {\it geometrically simply connected 4-manifold}. We treat more general class of 4-manifolds equipped with {\it geometrically isolated 2-handles}.
 
\begin{defn}\label{good}
Let $X$ be a compact 4-manifold with connected boundary $Y$. We say $X$ has a  \textit{geometrically isolated 2-handle} if there exists a 2-handle $h$ in a handle decomposition of $X$ such that 
\begin{enumerate}
    \item 
    $h$ does not run over any 1-handles in that handle decomposition and 
    \item
    $h$ generates a non-trivial element in $H_2(X;\Z)/\operatorname{Tor}$. 
\end{enumerate}
\end{defn}

As an application of the relative Bauer-Furuta invariant and a similar discussion in the proof of \cref{main1}, we put some constraints on the topology of symplectic caps. As we discussed before, there are several known constraints on the topology of symplectic fillings of a contact 3-manifold $(Y, \xi)$, but not much is known about the topology of symplectic caps. 
For example, $-\Sigma(2,3,5)$ does not admit any symplectic filling \cite{Li98} but a result of Etnyre--Honda \cite{ehcap} guarantees that it has infinitely many distinct symplectic caps for each contact structure.

\begin{thm}\label{cap1}
The following results hold: 
\begin{itemize}
    \item[(i)] 
The contact 3-manifold $(S^3,\xi_{std})$ does not have any positive definite symplectic cap having a geometrically isolated 2-handle such that $b_1=0$ and $b^+_2\geq 2$.   
\item[(ii)]  The contact 3-manifold $\Sigma(2,3,5)$ with its unique tight contact structure does not have any positive definite symplectic cap having a geometrically isolated 2-handle such that $b_1=0$, $b^+_2\geq 2$, and there is no 2-torsion in its homology.
\end{itemize}
\end{thm}

Note that $\C P^2 \setminus \operatorname{int} D^4$ gives a positive definite symplectic cap of $(S^3, \xi_{std})$, so the assumption $b_2^+(X)> 1$ is necessary. 

\begin{rem}
Notice that, since a geometrically simply-connected closed symplectic 4-manifold with $b^+_2 >1$ has a geometrically isolated 2-handle, after removing $(D^4, \om_{std})$, we get a positive definite symplectic cap to $(S^3,\xi_{std})$ which contradicts \cref{cap1}(i). Thus we can reprove that any positive definite geometrically simply connected closed 4-manifold with $b^+_2 >1$ does not admit a symplectic structure \cite[Theorem 1.1]{HL19}, \cite[Corollary 1.6]{Y19}. 
\end{rem}

We can generalize \cref{cap1} for any contact 3-manifold $(Y, \xi)$ having a metric with positive scalar curvature. 
\begin{thm}\label{cap}
Let $(Y, \xi)$ be a contact 3-manifold with a symplectic filling that has $b_1=0$. If  $b_1(Y)=0$ and $Y$ admits a positive scalar curvature metric, then $(Y,\xi)$ does not have positive definite symplectic cap $X$ with $b_1=0$ and $b^+_2\geq 2$ having a geometrically isolated 2-handle and a Spin$^c$ structure $\s_X$ satisfying
\[
\frac{-c^2_1(\s_X)+b_2(X)}{8}=\delta(Y, \s_\xi), 
\]
where $\delta(Y, \s)$ is the Fr\o yshov invariant of $(Y, \s)$, with the convention $\delta(\Sigma(2,3,5))=1$.
\end{thm}

Given this evidence, we propose the following conjecture.

\begin{conjecture}
A contact 3-manifold $(Y,\xi)$, where $Y$ is an $L$-space and $\xi$ is symplectically fillable, does not have a simply-connected positive definite symplectic cap with $b_2^+ >1$.
\end{conjecture}
Some of the techniques that have been used in the proofs of the results above can be further generalized to study codimension-0 embeddings of manifolds into closed symplectic 4-manifolds and their connected sums (see \cref{main1} and \cref{b=1appl} for the precise statements). 


 
\subsection{Study of exotic 4-manifolds with boundary} 
We now explore another direction of applications namely, the understanding exotic structures on 4-manifolds with connected boundary. 
Recently, the study of exotic structures on 4-manifolds with connected boundary has gotten a lot of attention. One interesting such problem is whether all 3-manifolds bound a 4-manifold that admits exotic structures. This problem has been studied by Yasui \cite{yasui} and independently Etnyre, Min and the second author~\cite{etnyre}, using versions of Seiberg--Witten theory and Heegaard Floer theory. The techniques we have developed in this paper can also be useful to study such exotic behaviour. 
\begin{thm}\label{exotic}
Let $Y$ be an oriented, connected, closed 3-manifold such that either $Y$ or $\overline{Y}$ admits a contact structure that has a weak filling with $b_3=0$. Then there exists a pair of compact oriented smooth 4-manifolds $X$ and $X'$ with boundary $Y$ such that $X$ and $X'$ are homeomorphic but not diffeomorphic.  
\end{thm}
The result was first proven by Yasui~\cite{yasui} and later independently by Etnyre, Min and the second author~\cite{etnyre}. Using the first author's invariant, we are presenting different proof for the case $b_3=0$ (although the original results do not require $b_3=0$). Our examples are different from the earlier ones \cite{etnyre, yasui} in a sense that the invariants developed in \cite{etnyre} and \cite{yasui} vanish on these examples. So those earlier techniques cannot detect the exotic behaviour of such 4-manifolds with boundary. Also, note that, in \cite{etnyre} and \cite{yasui}, infinitely many exotic 4-manifolds are constructed for a given boundary 3-manifold.

The structure of paper is as follows: in \cref{Pre}, we review several notions in contact geometry, symplectic geometry and Bauer--Furuta type refinement of Kronheimer--Mrowka's invariant that are used in the proof of our main theorems. In \cref{Adjunction}, following \cite{Fr05} and \cite{KM94} we give a proof of \cref{main2}: an adjunction type inequality for the first author's invariant \cite{I19}. In \cref{Adjunction2}, we prove an adjunction inequality for relative Bauer--Furuta invariant of 4-manifolds with homotopy L-space boundary. We follow the methods given in \cite{Fr05} and \cite{KM94} to prove such an adjunction inequality. In the same section, we also prove \cref{b=1}. 
Finally in \cref{Application}, we give proofs all the applications including \cref{h-slice}, \cref{top h}, \cref{main1}, \cref{b=1appl} and \cref{exotic}.

\begin{acknowledgement}
The authors are grateful to Ciprian Manolescu whose talk in the Regensburg low-dimensional geometry and topology seminar inspired this project. The authors would like to express their appreciation to Mikio Furuta for answering their questions on the relative Bauer--Furuta invariants.
The authors would like to express their deep gratitude to John Etnyre, Marco Golla, Maggie Miller and Marco Marengon 
for their kind and generous offer improving  this paper.
The authors also wish to thank Hokuto Konno, Ciprian Manolescu, Kouki Sato, O\u{g}uz \c{S}avk, Yuichi Yamada, Kouichi Yasui for giving many helpful comments on the drafts of the paper.  
The authors also thank Francesco Lin, Lisa Piccirillo and Ian Zemke for showing interest in this project. We are also grateful to anonymous referees for their helpful comments.  

The first author was supported by JSPS KAKENHI Grant Number 19J23048 and the Program for Leading Graduate Schools, MEXT, Japan. The second author was partially supported by NSF grant DMS-1906414.
The third author was supported by JSPS KAKENHI Grant Number 17H06461 and 20K22319 and RIKEN iTHEMS Program.
\end{acknowledgement}

\section{Preliminaries} \label{Pre}

\subsection{Contact Geometry}\label{contact}

Recall that a \textit{(co-orientable) contact  structure} $\xi$ on an oriented 3-manifold $Y$ is the kernel of 1-form $\theta \in \Omega ^1(Y)$ such that $\theta \wedge d\theta$ is positive. Darboux's theorem says that every contact 3-manifold $(Y,\xi)$ is locally contactomorphic to the standard contact structure on $\R^3$
\[
(\R ^3,\xi_{std}= \ker(dz - ydx)).
\] 
All orientable 3-manifolds admit contact structures. A knot $K \subset (Y,\xi)$ is called \textit{Legendrian} if at every point of $K$, the tangent line to $L$ lies in the contact plane at that point. A Legendrian knot $K$ in a contact manifold $(Y,\xi)$ has a standard neighborhood $N$ and a framing $fr_\xi$ given by the contact planes. If $K$ is null-homologous, then $fr_\xi$ relative to the Seifert framing is the \textit{Thurston--Bennequin} invariant of $K$, which is  denoted by $tb(K)$.  If one does $fr_\xi-1$-surgery on $K$ by removing $N$ and gluing back a solid torus so as to effect the desired surgery, then there is a unique way to extend $\xi|_{Y-N}$ over the surgery torus so that it is tight on the surgery torus. The resulting contact manifold is said to be obtained from $(Y,\xi)$ by \textit{Legendrian surgery} on $K$. 
  
  A Legendrian knot $K$ in $(\R^3, \xi_{std})$ projects to a closed curve $\gamma$ in the $xz$--plane, which is known as a \textit{front projection} of $K$. The curve $\gamma$ uniquely determines the Legendrian knot $K$ which can be reconstructed by setting $y(t)$ as the slope of $\gamma (t)$. Thus, at a crossing, the most negative slope curves always stay at front. There are two types of possible cusp singularities corresponding to points where $dz/dx =0$ which are called {\it left cusps} and {\it right cusps}. From the front diagram, we can compute the Thurston--Bennequin invariant by the formula 
 \[
 tb(K)= \operatorname{Writhe}(K)- \#\{\text{Left cusps}\}.
 \]
 For a knot $K$ in $\R^3$, we define the \textit{maximum Thurston--Bennequin number} of $K$ by
  \[
  TB(K)=\  \max\set {tb(K') | K'\text{:Legendrian representation of} \ K\ \text{ in }  (\R^3,\xi_{std}) } .
  \]
  
  Since $\R^3$ (or $S^3$) has the unique tight contact structure up to isotopy proven by Eliashberg \cite{utight}, $TB(K)$ is an invariant of smooth knot type.
\subsection{Symplectic Geometry}\label{symplectic}  Let $(Y, \xi)$ be a closed oriented 3-manifold equipped with a contact structure.
  We recall that a compact symplectic manifold $(X,\omega)$ is a \textit{strong  symplectic  filling} of $(Y,\xi)$ if $\partial X=Y$ and there is a vector field $v$ defined near $\partial X$ such that the Lie derivative of $\omega$ satisfies $\mathcal{L}_v \omega= \omega$, $v$ points out of $X$ and $\iota_v\omega$ is a contact form for $\xi$. Moreover, $(X,\omega)$ is a {\it strong symplectic cap} for $(Y,\xi)$ if it satisfies all the properties above, except $\partial X=-Y$ and $v$ points into $X$. We also say $(X,\omega)$ is a \textit{weak symplectic filling} of $(Y,\xi)$ if $\partial X=Y$ and $\omega|_\xi>0$  (here all our contact structures are co-oriented). Similarly, $(X,\omega)$ is a {\it weak symplectic cap} of $(Y,\xi)$ if $\partial X=-Y$ and $\omega|_\xi>0$. We shall say that $(Y,\xi)$ is (strongly or weakly) \textit{semi-fillable} if there is a connected (strong or weak) filling $(X,\omega)$ whose one boundary component is $(Y, \xi)$. 
  $(Y,\xi)$ is (strongly or weakly) \textit{fillable} if there is a connected (strong or weak) filling of it.
  
  While symplectic fillings do not necessarily exist for a contact 3-manifold,  Etnyre and Honda \cite{ehcap} showed that there always exist symplectic caps for any contact 3-manifolds.
  This in particular implies that semi-fillability is equivalent to fillablity.

A \textit{symplectic  cobordism} from the contact manifold $(Y_-,\xi_-)$ to $(Y_+,\xi_+)$ is a compact symplectic manifold $(W,\omega)$ with boundary $-Y_-\cup Y_+$ where $Y_-$ is a \textit{concave} boundary component and $Y_+$ is \textit{convex}, this means that there is a vector field $v$ near $\partial W$ which points transversally inwards at $Y_-$ and transversally outwards at $Y_+$, $\mathcal{L}_v \omega= \omega$ and $\iota_v \om |_{Y_\pm}$ is a contact form of $\xi_\pm $. The first result we will need concerns when symplectic cobordisms can be glued together. If we have two symplectic cobordisms $(W_1, \om_1)$ and $(W_2, \om_2)$ with $\partial W_1 = -Y_0\cup Y_1$ and $\partial W_2 = -Y_1\cup Y_2$,
we can glue the symplectic forms so that the glued manifold $W:= W_1 \cup_{Y_1} W_2$ a structure of symplectic cobordism with concave component $Y_0$ and convex component $Y_2$.

Another way to build cobordisms is by {\it Weinstein handle attachment} \cite{ Weinstein91}. One may attach a 0, 1, or 2-handle to the convex end of a symplectic cobordism to get a new symplectic cobordism with the new convex end described as follows. For a 0-handle attachment, one merely forms the disjoint union with a standard 4--ball and so the new convex boundary will be the old boundary disjoint union with the standard contact structure on $S^3$. For a 1-handle attachment, the convex boundary undergoes, possibly internal, a connected sum. A 2-handle is attached along a Legendrian knot $L$ with framing one less that the contact framing, and the convex boundary undergoes a Legendrian surgery. 

\begin{thm}\label{cob}
Given a contact 3-manifold $(Y,\xi=\ker \theta)$ let $W$ be a part of its symplectization, that is $(W= [0,1]\times Y, \omega= d(e^t\theta))$. Let $L$ be a Legendrian knot in $(Y,\xi)$ where we think of $Y$ as $Y\times \{ 1 \}$. If $W'$ is obtained from $W$ by attaching a 2-handle along $L$ with framing one less than the contact framing, then the upper boundary $(Y', \xi ')$ is still a convex boundary. Moreover, if the 2-handle is attached to a strong symplectic filling (respectively weak filling) of $(Y,\xi)$ then the resultant manifold would be a strong symplectic filling (respectively weak filling) of $(Y'\xi ')$.
\end{thm}

The theorem for Stein fillings was proven by Eliashberg \cite{yasha91}, for strong fillings by Weinstein \cite{Weinstein91}, and was first stated for weak fillings by Etnyre and Honda \cite{eh02}.

\subsection{Bauer--Furuta version of Kronheimer--Mrowka's invariant for 4-manifolds with contact boundary}\label{finite dim}
In this subsection, we summarize the definition of the Bauer--Furuta version of Kronheimer--Mrowka's invariant 
 $\Psi(W, \xi , \s_W)$
 introduced by the first author \cite{I19}. The geometric setting is the same as that given in \cite{KM97} except for the condition $b_3(W) = \dim H^1(W, \partial W; \R)=0$. 
\par
Let $W$ be a compact oriented 4-manifold with nonempty boundary.
We assume $H^1(W, \partial W; \R)=0$, in particular, $Y=\partial W$ is connected. 
Let $\xi$ be a contact structure on $Y=\partial W$ compatible with the boundary orientation.

Pick a contact 1-form $\theta$ on $Y$ and a complex structure $J$ of $\xi$ compatible with the orientation.
There is now an unique Riemannian metric $g_1$ on $Y$ such that $\theta$ satisfies that $|\theta|=1$, $d\theta=2*\theta$, and $J$ is an isometry for $g|_\xi$, where $*$ is the Hodge star operator with respect to $g_1$.
Define a symplectic form $\omega_0$ on $\R^{\geq 1}\times Y$ by the formula
$\omega_0=\frac{1}{2}d(s^2\theta)$, where $s$ is the coordinate of $\R^{\geq 1}$.
We define a conical metric on $\R^{\geq 1} \times Y$ by 
\begin{align}\label{conical metric} 
g_0 := ds^2 + s^2 g_1. 
 \end{align}

On $\R^{\geq 1} \times Y$, we have a canonical $\Spinc$ structure $\s_0$, a canonical $\Spinc$ connection $A_0$, a canonical positive Spinor $\Phi_0$.
These are given as follows.
The pair $(g_0, \om_0)$ determines an almost complex structure $J$ on $\R^{\geq 1} \times Y$. This defines a $\Spinc$ structure on $\R^{\geq 1} \times Y$:
 \[
 \s_0:= ( S^+ = \Lambda^{0,0}_J \oplus  \Lambda_J^{0,2}, S^- = \Lambda_J^{0,1} , \rho :\Lambda^1 \to \Hom (S^+ , S^-) ),
 \]
  where  
\[
 \rho  = \sqrt{2} \operatorname{Symbol} (\overline{\partial} + \overline{\partial}^* ) . 
 \]
 (See Lemma 2.1 in \cite{KM97}.)
 The notation $\Phi_0$ denotes 
 \[
 (1,0) \in \Om_{\R^{\geq1} \times Y}^{0,0} \oplus  \Om_{\R^{\geq1} \times Y}^{0,2}= \Gamma (S^+|_{\R^{\geq1} \times Y}).
 \]
 Then the {\it canonical $\Spinc$ connection} $A_0$ on $\s_0$ is uniquely defined by the equation 
 \begin{align}\label{A0}
 D^+_{A_0} \Phi_0= 0
 \end{align}
 on $\R^{\geq 1} \times Y$.

Let $W^+$ be a non-compact 4-manifold with conical end 
\[
W^+ := W \cup_Y  (\R^{\geq 1} \times Y).
\]
Pick a Riemann metric $g_{W^+}$ on $W^+$ such that $g_{W^+}|_{\R^{\geq 1} \times Y}=g_0$. 
Fix a $\Spinc$ structure $\s_{W^+}=(S^{\pm}_{W^+}, \rho_{W^+}) $ on $W^+$ equipped with an isomorphism $\s_{W^+} \to \s_0$ on $W^+\setminus W$.
We will omit this isomorphism in our notation.
Fix a smooth extension of $(A_0, \Phi_0)$ on $W^+$.
We also fix a nowhere zero proper extension $\sigma$ of $s \in \R^{\geq 1}$ coordinate to all of $W^+$ which is $0$ on $W \setminus \nu (\partial W)$, where $\nu (\partial W)$ is a small color neighborhood of $\partial W$ in $W$.

On $W^+$, {\it weighted Sobolev spaces} 
\[
\widehat{\cU}_{W^+}=L^2_{k, \alpha, A_0 }(i \Lambda^1_{W^+}\oplus S^+_{W^+}) \text{ and } 
\]
\[
\widehat{\cV}_{W^+}=L^2_{k-1, \alpha, A_0 }(i\Lambda^0_{W^+}\oplus i\Lambda^+_{W^+}\oplus S^-_{W^+})
\]
are defined using $\sigma$ for a positive real number $\al \in \R$ and $k \geq 4$, where $S^+_{W^+}$ and $S^-_{W^+}$ are positive and negative spinor bundles and the Sobolev spaces are given as completions of the following inner products: 
\begin{align}\label{inner}
\langle s_1, s_2\rangle_{L^2_{k, \alpha, A} } := \sum_{i=0}^k \int_{W^+} e^{2\alpha \sigma} \langle \nabla^i_A s_1, \nabla^i_A s_2  \rangle  \operatorname{dvol}_{W^+}, 
\end{align}
where the connection $\nabla^i_A$ is the induced connection from $A$ and the Levi-Civita connection.

Fix a sufficiently small positive real number $\al$. 
The invariant $\Psi(W, \xi, \s_{W, \xi})$(\cite{I19}) is obtained as a finite-dimensional approximation of the Seiberg--Witten map
\begin{align}\label{FX+}
\begin{split}
&\widehat{\mathcal{F}}_{W^+}:  \widehat{\cU}_{W^+}\to \widehat{\cV}_{W^+}\\
&(a, \phi)\mapsto (d^{*_\alpha}a, d^+a-\rho^{-1}(\phi\Phi^*_0+\Phi_0\phi^*)_0-\rho^{-1}(\phi\phi^*)_0, D^+_{A_0}\phi+\rho(a)\Phi_0+\rho(a)\phi). 
\end{split}
\end{align}

\par
The finite-dimensional approximation goes as follows.
We decompose $\widehat{\cF}_{W^+}$ as $\widehat{L}_{W^+}+\widehat{C}_{W^+}$ where 
\[
 \widehat{L}_{W^+}(a, \phi)=(d^{*_\alpha}a, d^+a-\rho^{-1}(\phi\Phi^*_0+\Phi_0\phi^*)_0, D^+_{A_0}\phi+ \rho(a)\Phi_0)
 \]
 and 
 \[
  \widehat{C}_{W^+}(a, \phi) = (0, -\rho^{-1}(\phi\phi^*)_0, \rho(a)\phi). 
  \]
Then $\widehat{L}_{W^+}$ is linear Fredholm and  $\widehat{C}_{W^+}$ is quadratic, compact. (Here we used $\alpha>0$. )
Pick an increasing sequence of finite-dimensional subspaces $\widehat{\cV}_{W^+, n} \subset \cV_{W^+}\, (n \in \Z^{\geq 1})$ such that 
\begin{itemize}
    \item For any $\gamma \in \widehat{\cV}_{W^+}$, 
    \[
 \| \pr_{\widehat{\cV}_{W^+, n}}( \gamma) - \gamma \|_{\widehat{\cV}_{W^+}}   \to 0  \text{ as } n\to \infty
\]
and 
\item $\Coker \widehat{L}_{W^+} := (\im \widehat{L}_{W^+})^{\perp_{L^2_{k-1, \alpha}}} \subset \widehat{\cV}_{W^+, 1}$.  
\end{itemize}
Let 
\[
\widehat{\cU}_{W^+, n}=\wh{L}^{-1}(\widehat{\cV}_{W^+, n})\subset \wh{\cU}_{W^+},
\]
 and 
\[
\mathcal{F}_{W^+, n}:=  \pr_{\widehat{\cV}_{W^+, n}}   \circ \mathcal{F}_{W^+}: \widehat{\cU}_{W^+, n}\to \widehat{\cV}_{W^+, n}. 
\]
We can show that for a large $R>0$, a small $\varepsilon$ and a large $n$, we have a well-defined map
\[
\mathcal{F}_{W^+, n}: B(\widehat{\cU}_{W^+, n}, R)/S(\widehat{\cU}_{W^+, n}, R)\to B(\widehat{\cV}_{W^+, n}, \varepsilon)/S(\widehat{\cV}_{W^+, n}, \varepsilon).
\]
The stable homotopy class of $\mathcal{F}_{W^+, n}$ defines {\it the Bauer--Furuta version of Kronheimer--Mrowka's invariant} 
\begin{equation}\label{iida}
\Psi(W, \xi , \s_W) \in \pi^S_{\langle e(S^+_W, \Phi_0), [(W, \partial W)]\rangle} 
\end{equation}
 defined in \cite{I19}, where $\pi^S_{i}$ is the $i$-th stable homotopy group of the spheres and $e(S^+_W, \Phi_0)$ is the relative Euler class of $S^+_W$ with respect to the section $\Phi_0|_Y $.
\par
We will review some properties of Bauer-Furuta invariant and this invariant.
First, let us see non-vanishing results for Bauer-Furuta invariant. 

\begin{thm}[Bauer, \cite{B04}]\label{bauer} 
 We consider the following two types of $\Spinc$-4-manifolds:
\begin{itemize}
    \item $(X, \s)$ is closed sympectic 4-manifolds with $b_1(X)=0$ and $b^+_2(X) \equiv 3 \operatorname{mod} 4$, and
    \item $(X', \s')$ is a closed negative definite $\Spinc$ 4-manifold  with $b_1(X')=0$ and $d(\s')=-1$.
\end{itemize}
Then as non-equivariant stable homotopy classes of maps, $BF (X, \s)$ is a generator of the 1-st stable homotopy group $\pi^S_1 \cong \Z_2$ and $BF (X', \s')$ is a generator of the $0$-th stable homotopy group $\pi^S_0 \cong \Z$.

\end{thm}

Using \cref{bauer} and the connected sum formula below, we can prove a certain non-vanishing result. 

\begin{thm}[Iida, \cite{I19}]\label{conn sum}
Let $(W, \s_W)$ be an oriented $\Spinc$ compact 4-manifold whose boundary is a contact $3$-manifold $(Y, \xi)$ with $b_3(W)=0$, $\s_W|_Y = \s_\xi$ and let $(X, \s_X)$ be a closed $\Spinc$ 4-manifold with $b_1(X)=0$. Then, we have 
\begin{align}
     \Psi (W\# X , \xi, \s_W \# \s_X ) =  \Psi (W, \xi, \s_W ) \wedge BF (X, \s_X ) 
\end{align}
in the stable homotopy group up to sign. Here we forget the $ S^1$ action of $BF (X, \s_X )$. 

\end{thm}

Since Iida's invariant \eqref{iida} is $\pm\id$ for any weak symplectic filling with $b_3=0$ \cite[Corollary  4.3]{I19}, thus by combining \cref{bauer}, \cref{conn sum}, we obtain the following non-vanishing results. 
\begin{thm}[Iida, \cite{I19}]\label{main0}
Let $(W, \om)$ be a weak symplectic filling of a contact 3-manifold $(Y,\xi)$ with $b_3(W)=0$. We consider the following two types of $\Spinc$-4-manifolds.
\begin{itemize}
    \item $(X, \om_1)$, $(X_2, \om_2)$ and $(X_3, \om_3)$ are closed sympectic 4-manifolds with $b_1(X_i)=0$ and $b^+_2(X_i) \equiv 3 \operatorname{mod} 4$, and
    \item $(X, \s)$ is a closed negative definite $\Spinc$ 4-manifold  with $b_1(X)=0$ and $d(\s)=-1$.
\end{itemize}
Then,  the invariant $\Psi (X', \xi, \s' )$ does not vanish for 
\[
(X',\s') = \begin{cases}  (W \# X, \s_{\om} \# \s)  \\
(W\# X \# X_1, \s_{\om} \# \s\# \s_{\om_1} ) \\ 
(W\# X \# X_1\# X_2, \s_{\om} \# \s\# \s_{\om_1}\# \s_{\om_2} ) \\
(W\# X \# X_1\# X_2\# X_3, \s_{\om} \# \s\# \s_{\om_1}\# \s_{\om_2} \# \s_{\om_3}) . \\
\end{cases} 
\]
\qed
\end{thm}

\section{Adjunction-Inequality for Iida's invariant}\label{Adjunction}
In this section, we will give a proof of \cref{main2} by proving a more general result as \cref{main2'}. 
There are several proofs of adjunction inequalities for the usual Seiberg--Witten invariants of closed 4-manifolds \cite{ KM94, FS95, Fr05, KM07}. In this paper, we follow the method given in \cite{Fr05} and \cite{KM94}. 

\begin{thm}\label{main2'}
Let $(W, \s)$ be an oriented Spin$^c$ compact 4-manifold whose boundary is a contact $3$-manifold $(Y, \xi)$ with $b_3(W)=0$ and $\s|_Y = \s_\xi$.
If $\Psi (W, \xi, \s) \neq 0$, then the following results hold: 
\begin{itemize}
\item[(i)]
There are no closed oriented 3-dimensional submanifolds $M \subset \operatorname{int}(W)$ satisfying the following three conditions.
\begin{enumerate}
\item
$M$ admits a Riemann metric with positive scalar curvature.
\item
$H^2(W, \partial W; \R)\to H^2(M; \R)$ is non-zero.
\item $M$ separates $W$.
\end{enumerate}
\item[(ii)] Non-torsion homology class in $H_2(W, \partial W; \Z)$ cannot be realized by an embedded 2-sphere whose self-intersection number is non-negative. 
\item[(iii)] For any connected, orientable, embedded, closed surface $\Sigma \subset W$ with $g(\Sigma)>0$ and $[\Sigma]\cdot [\Sigma] \geq 0$, we have 
\[
| \langle c_1 (\s) , [\Sigma] \rangle | + [\Sigma]\cdot [\Sigma]\leq 2 g(\Sigma) - 2. 
\]
\end{itemize}
\end{thm}

The first part of \cref{main2} is contained in \cref{main2'}(ii). 
\begin{rem}
Note that (ii) follows from (i) by considering a normal sphere bundle of the embedded 2-sphere. 
\end{rem}

\subsection{Proof of \cref{main2'}(i)} 
First, as in the previous section, we fix the data $g_{W^+}$, $\s_{W^+}$, $S^\pm_{W^+}$, $A_0$, $\Phi_0$, $\widehat{\mathcal{U}}_{W^+}$ and $\widehat{\mathcal{V}}_{W^+}$. 
In addition, suppose there is an embedded codimension-1 manifold  $M \subset \text{int}(W)$. We can suppose that such data $g_{W^+}$, $\s_{W, \xi}$, $S^\pm_{W^+}$, $A_0$, and $\Phi_0$ are translation invariant on some product neighborhood on $M$. 
We set $\overline{W}^+  :=W ^+\setminus \overset{\circ}{\nu}(M)  $, where $\overset{\circ}{\nu} (M)$ is an even smaller open tubular neighborhood of $M$. Then $\partial \overline{W}^+  = M \cup (-M)$.

Let $T$ be a non-negative real number and $W^+(T)$ be a Riemannian manifold obtained from $W^+$ by inserting a neck $[-T, T]\times M$ with a product metric: 
\[
W^+(T)=  [-T,  T]\times M \cup_{ M \cup (-M) } \overline{W}^+  
\]
such that the restriction of the metric of $W^+(T)$ on $\overline{W}^+ $ coincides with a given metric. We set $W(T):= W^+(T) \setminus \R^{\geq 1}\times Y$.
Note that $W^+(T)$ and $W^+$ are diffeomorphic as manifolds. We will identify $W^+(0)$ with $W^+$ as Riemann manifolds. 
We fix the obvious  $\Spinc$ structures and reference configuration $(A_0, \Phi_0)$ on $W^+(T)$ which are product on the smaller neck $[-T, T]\times M $ and coincide outside the neck with those of $W^+ (0)=W^+$.
We define $\sigma : W^+(T)\to \R$ to be $0$ on the neck and to coincide outside the neck with those of $W^+ (0)=W^+$.
On $W^+(T)$, weighted Sobolev spaces 
\[
\widehat{\cU}_{W^+(T)}=L^2_{k, \alpha}(i \Lambda^1_{W^+(T)}\oplus S^+_{W^+(T)}) \text{ and } 
\]
\[
\widehat{\cV}_{W^+(T)}=L^2_{k-1, \alpha}(i\Lambda^0_{W^+(T)}\oplus i\Lambda^+_{W^+(T)}\oplus S^-_{W^+(T)})
\]
are defined as before for a positive real number $\al \in \R$ and $k \geq 4$. Fix a sufficiently small positive real number $\alpha$. 
We also have a family of Seiberg--Witten maps 
\[
\widehat{\mathcal{F}}_{W^+(T)} = \widehat{L}_{W^+(T)} + \widehat{C}_{W^+(T)} : \widehat{\cU}_{W^+(T)} \to \widehat{\cV}_{W^+(T)} . 
\]

\par 
Since we have the condition 
\[
 \im (H^2(W, \partial W; \R)\to H^2(M; \R) ) \neq 0,  
 \]
 we can take a non-exact closed 2-form $\eta$ on $M$ and a 2-form $\hat{\eta}$ on $W^+(T)$ whose support is contained in $W^+(T)\setminus \R^{\geq 1}\times Y$ and which extends the 2-form on the neck which is the pull-back of  $\eta$ by the projection $[-T, T]\times M\to M$.

For each $T \geq 0$, define the $\eta$-perturbed 
Seiberg--Witten map
\begin{align}\label{FX+1}
\begin{split}
&\widehat{\mathcal{F}}_{W^+(T), \eta}:  \widehat{\cU}_{W^+(T)}\to \widehat{\cV}_{W^+(T)}\\
&(a, \phi)\mapsto (d^{*_\alpha}a, d^+a-\rho^{-1}(\phi\Phi^*_0+\Phi_0\phi^*)_0-\rho^{-1}(\phi\phi^*)_0+\hat{\eta}^+, D^+_{A_0}\phi+\rho(a)\Phi_0+\rho(a)\phi). 
\end{split}
\end{align}
and the Seiberg--Witten moduli space for the $\eta$-perturbed equation by 
\[
\mathcal{M}_{\eta}(T): =\widehat{\mathcal{F}}_{W^+(T), \eta}^{-1} (0) 
\]
For any $T$, we fix an increasing sequence of finite dimensional vector subspaces $\widehat{\cV}_{W^+(T), n}$ of $\widehat{\cV}_{W^+(T)}$ such that 
we have a well-defined map 
\begin{align}\label{BF}
f_T : B(\widehat{\cU}_{W^+(T), n_T}; R_T )/ S(\widehat{\cU}_{W^+(T), n_T}; R_T ) \to  B(\widehat{\cV}_{W^+(T), n_T}; \varepsilon_T )/ S(\widehat{\cV}_{W^+(T), n_T}; \varepsilon_T ), 
\end{align}
as a finite dimensional approximation of $\widehat{\mathcal{F}}_{W^+(T), \eta}$, 
where 
\[
\widehat{\cU}_{W^+(T), n} := \widehat{L}_{W^+(T)} ^{-1} \widehat{\cV}_{W^+(T), n}
\]
for sufficiently large $R_T$, $n_T$ and a sufficiently small $\varepsilon_T$.

Standard arguments show that $\eqref{BF}$ represents the homotopy class $\Psi(W, \xi, \s)$ by finite dimensional approximation of this perturbed equation.
\begin{lem} \label{non-empty}If $\Psi(W, \xi, \s)$ is not stably null-homotopic, then 
$\mathcal{M}_{\eta}(T)$ is non-empty for all $T \geq 0$
\end{lem}
\begin{proof}
We fix an arbitary $T$.

Finite dimensional approximation gives a sequence of  maps
\[
\widehat{L}_{W^+(T)}+\pr_{\widehat{\cV}_n}\widehat{C}_{W^+(T)}: B(\widehat{\cU}_{W^+(T), n}; R_T)/ S(\widehat{\cU}_{W^+(T), n}; R_T)
\]
\[
\to B(\widehat{\cV}_{W^+(T), n}, \varepsilon_T)/S(\widehat{\cV}_{W^+(T), n}, \varepsilon_T).
\]
The assumption implies that, for sufficiently large $n$,
these maps are not null homotopic, and in particular surjective.
Thus, we can take a sequence $\gamma_n \in \widehat{\cU}_{W^+(T)} $ such that 
\[
\|\gamma_n\|_{L^2_{k, \alpha}} =R_T \text{ and }(\widehat{L}_{W^+(T)}+\pr_{\widehat{\cV}_{W^+(T), n}}\widehat{C}_{W^+(T)})(\gamma_n)=0.
\]
Since $\gamma_n$ is a bounded sequence,  we can take a weakly convergent subsequence. 
We denote the weak limit by $\gamma_\infty \in \widehat{\cU}_{W^+(T)} $.
We can also assume that $\widehat{C}_{W^+(T)}(\gamma_n)\to \widehat{C}_{W^+(T)}(\gamma_\infty)$ as $n \to \infty$ with fixed $T$
by taking a subsequence since
$\widehat{C}_{W^+(T)}: L^2_{k, \alpha}\to  L^2_{k-1, \alpha}$ is compact for $\alpha>0$, here we used Sobolev multiplication theorem for 4-manifolds with conical ends, see \cite[Lemma 2.1]{I19}. 
We have
\[
\begin{split}
(\widehat{L}_{W^+(T)}+\widehat{C}_{W^+(T)})(\gamma_n)
&=
(\widehat{L}_{W^+(T)}+\pr_{\widehat{\cV}_n}\widehat{C}_{W^+(T)})(\gamma_n)+(1-\pr_{\widehat{\cV}_{W^+(T), n}})C(\gamma_n)\\
&=(1-\pr_{\widehat{\cV}_{W^+(T), n}})C(\gamma_n) \to 0\quad \text{as }n\to \infty . 
\end{split}
\]
Thus, $\gamma_\infty$ satisfies
\[
(\widehat{L}_{W^+(T)}+\widehat{C}_{W^+(T)})(\gamma_\infty)=0.
\]
In other words, $\gamma_\infty \in \mathcal{M}_{\eta}(T)$. This completes the proof. 
\end{proof}
\par
Using \cref{non-empty}, we take a 1-parameter family of solutions $(A_T,\Phi_T) \in \mathcal{M}_{\eta}(T)$ for each $T\geq 0$.
In order to show \cref{main2} (i), we prove the following proposition as a preliminary, which is essentially proved in \cite[Proposition 8]{KM94}, see also \cite[Proposition 4.2.4]{MR06}.
\begin{prop} \label{existence}
Suppose $\mathcal{M}_{\eta}(T)\neq \emptyset$ for all sufficiently large $T$.
Consider the $\Spinc$ structure on the cylinder $\R\times M$ 
obtained as the pull back of $\s|_M$ by the projection $\R\times M\to M$ 

Then, there exists a solution to the $\eta$-perturbed Seiberg--Witten equation
\begin{equation}\label{SWcyl}
\begin{split}
\frac{1}{2}F^+_{A^t}-\rho^{-1}(\Phi\Phi^*)_0&=2i\hat{\eta}^+\\
D^+_A\Phi&=0
\end{split}
\end{equation}
 which is translation invariant and in a temporal gauge.
Here, $\hat{\eta}$ is the pull back of $\eta$ by the projection $\R\times M\to M$ and the word temporal gauge means that the $dt$-component of the $\Spinc$ connection vanishes.
\end{prop}
\begin{proof}
The proof is essentially the same as that of \cite[Proposition 8]{KM94}.
Suppose we have $[A_T, \Phi_T] \in \mathcal{M}_{\eta}(T)$ for each sufficiently large $T$. 
Recall that for a pair $(B, \Psi)$ where $B$ is a $\Spinc$ connection and $\Psi$ is a Spinor of $\s|_M=(S_M, \rho_M)$, the $\eta$-perturbed Chern-Simons-Dirac functional is defined by
\[
\cL_\eta(B, \Psi)=-\frac{1}{8}\int_M ({B^t}-{B^t_0})\wedge (F_{B^t}+F_{B^t_0})+\int_M (B^t-B^t_0) \wedge (i\eta)+\frac{1}{2}\int_M \langle \Psi, D_B \Psi\rangle \dvol_M.
\]
Note that 
\[
\cL_\eta(u\cdot(B, \Psi))-\cL_\eta(B, \Psi)=\langle (2\pi^2 c_1 (S_M) - 4\pi[\eta] ) \cup [u], [M]\rangle , 
\]
where $u: M\to S^1$ is a gauge transformation and 
\[
[u]=\frac{1}{2\pi i} u^{-1}du 
\] 
is an element of $H^1(M; \Z)$ determined by $u$. 
The following boundedness result is a key lemma of our proof: 
\begin{lem}\label{bounded} Let $(A_T, \Phi_T)$ be an element in $\mathcal{M}_{\eta}(T)$ for each $T \geq 0$.
Then the difference 
\[
 \cL_\eta((A_T, \Phi_T)|_{\{-T\}\times M})- \cL_\eta((A_T, \Phi_T)|_{\{T\}\times M})
 \] 
 is bounded by a constant which is independent of $T$.
\end{lem}

\begin{proof}
Since the perturbation here is compactly supported, from \cite{KM97},  there is a gauge transformation $u^T$ on $W^+(T)$ such that $u^T \cdot (A_T, \Phi_T)-(A_0, \Phi_0)$, its first derivatives are uniformly bounded with respect to $T$, $d u^T \in L^2_{k, \alpha}(W^+(T))$ and $\lim_{s\to \infty} u(s, y)  =1 $ for any $y \in Y$.
Since $2\pi^2c_1(\s_M)-4\pi[\eta]$ is obtained as a restriction of a cohomology class on $W^+(T)$, 
we have
\[
\cL_\eta(u^T\cdot(A_T, \Phi_T)|_{\{\pm T\}\times M})-\cL_\eta((A_T, \Phi_T)|_{\{\pm T\}\times M}) 
\]
\[
=\langle (2\pi^2 c_1 (S_M) - 4\pi\eta ) \cup [u^T|_{\{\pm T\}\times M}], [M]\rangle .
\]
Now we prove 
\[
\langle (2\pi^2 c_1 (S_M) - 4\pi\eta ) \cup [u^T|_{\{\pm T\}\times M}], [M]\rangle=0 . 
\]
Note that since $u^T$ converges to $1$ on the end of $W^+(T)$, $[u_T]$ belongs to $H^1( W(T), \partial W(T); \Z)$. Thus, $(2\pi^2 c_1 (S_M) - 4\pi\eta ) \cup [u^T|_{\{\pm T\}\times M}]$ is equal to the restriction of the class $(2\pi^2 c_1 (S^+_{W(T)}) - 4\pi\hat{\eta} ) \cup [u^T] \in H^3 ( W(T), \partial W(T); \Z)$.  
So the following commutative diagram implies the desired equation: 
\[
  \begin{CD}
     H_1(W(T); \Z)  @>{\partial=0}>> H_0(\{T\} \times M; \Z) \\
  @V{PD}V{\cong}V    @V{PD}V{\cong}V \\
     H^3 ( W(T), \partial W(T); \Z)    @>>> H^3(\{T\} \times M; \Z)  . 
  \end{CD}
\]
(The case of $-T$ is similar.)
Here, the top row  
\[
\partial: H_1(W(T); \Z)\to H_0(\{T\} \times M; \Z)
\]
is the Mayer-Vietoris connecting homomorphism and it vanishes because of the assumption that $M$ separates $W$.
\end{proof}
Now, we consider a 1-parameter family of solutions $(A_T, \Phi_T)|_{[-T+1, T-1]\times M} $ for $T \geq 2$. 
Then, by applying the argument of the end of the proof of \cite[Proposition 8]{KM94}, we obtain a solution to the $\eta$-perturbed Seiberg--Witten equation on $\R \times M$ which is translation invariant in a temporal gauge. 
This completes the proof. 
\end{proof}

By applying \cref{existence}, we obtain a translation invariant solution to \eqref{SWcyl} in a temporal gauge. 
However, this is a contradiction to the following result. 
\begin{prop}\label{non-ex}
Let $M$ be a closed oriented 3-manifold with a positive scalar curvature metric and a non-exact closed 1-form $\eta$. Then, for any  $\Spinc$ structure $\s_M$ on $M$, we can find a real number $s$ such that  $s\eta$-perturbed version of the equations \eqref{SWcyl} have no translation-invariant solutions in temporal gauge.
\end{prop}
\begin{proof}
Since $\eta$ is non-exact, there exists a positive real number $s_0>0$ such that for any real number $s$ satisfying $0<s\leq s_0$, 
\begin{equation}\label{c1eta}
c_1 (S_M) +\frac{2s}{\pi}[\eta] \neq 0 \in H^2(M; \R).
\end{equation}
Suppose there exists a translation  invariant solution $(A, \Phi)$ in temporal gauge for $(s\eta)$-perturbed version of \eqref{SWcyl}. 
Then, $(B, \Psi)=(A, \Phi)|_{\{t\}\times M}$ satisfies
\begin{equation}\label{3dimeq}
\begin{split}
\frac{1}{2}\rho(F_{B^t}-4is\eta)-(\Psi\Psi^*)_0&=0\\
D_B\Psi&=0.
\end{split}
\end{equation}
From the Weitenb\"ock formula, we obtain
\[
\begin{split}
\Delta |\Psi|^2
&=2\text{Re}\langle \Psi, \nabla^*_B \nabla_B \Psi\rangle-| \nabla_B \Psi|^2\\
&\leq 2\text{Re}\langle \Psi, \nabla^*_B \nabla_B \Psi\rangle\\
&=2\text{Re}\langle \Psi, D^2_B\Psi\rangle-\langle\Psi, \rho(F_{B^t})\Psi\rangle-\frac{\text{Scal}}{2}|\Psi|^2\\
&=-2\langle\Psi,  ((\Psi\Psi^*)_0+2s\rho(i\eta))\Psi\rangle-\frac{\text{Scal}}{2}|\Psi|^2\\
&=-|\Psi|^4 -4s \langle \Psi, \rho(i\eta)\Psi\rangle-\frac{\text{Scal}}{2}|\Psi|^2\\
&\leq -\left(|\Psi|^2+\frac{\text{Scal}}{2}-4s \|\eta\|_{C^0(M)}\right)|\Psi|^2.
\end{split}
\]
We conclude that  $\Psi\equiv 0$ on $M$ for $s$ sufficiently small, otherwise $|\Psi|^2$ achieves a local maximum at some point $p \in M$ and $\Delta |\Psi|^2(p)\geq 0$ holds, which implies
\[
0<|\Psi|^2(p)\leq -\frac{\text{Scal}(p)}{2}+4s \|\eta\|_{C^0(M)}
\]
and contradicts the positive scalar curvature assumption.
Thus the equation \eqref{3dimeq} implies
\[
F_{B^t}-4is\eta=0
\]
which implies
\[
\frac{i}{2\pi}F_{B^t}+\frac{2s}{\pi}\eta=0
\]
and that contradicts \eqref{c1eta}.
\end{proof}
This completes the proof of \cref{main2}(i). 
 
 \subsection{Proof of \cref{main2'}(iii)} 
 By considering blow-up, we can assume that a connected orientable embedded closed surface $\Sigma \subset W$ satisfies $g(\Sigma)>0$ and $[\Sigma]\cdot [\Sigma] =  0$. 
 Then we need to prove 
 \[
 | \langle c_1(S), [\Sigma] \rangle | \leq 2 (g (\Sigma) -1 ) . 
 \]
  
 We take a tubular neighborhood $\nu( \Sigma) = \Sigma \times D^2$ of $\Sigma$. Put $M := \partial  (\nu (\Sigma) ) \subset W$. First, as in the previous section, we fix data $\s_{W^+}$, $S^\pm_{W^+}$, $A_0$, $\Phi_0$, $\mathcal{U}_{W^+}$ and $\mathcal{V}_{W^+}$. Similarly, we consider the same thing $W^+(T)$, $\widehat{\mathcal{F}}_{W^+(T), \eta}:  \widehat{\cU}_{W^+(T)}\to \widehat{\cV}_{W^+(T)}$ and 
 \begin{align}\label{BF1}
 \begin{split}
f_T : B(\widehat{\cU}_{W^+(T), n_T}; R_T )/ S(\widehat{\cU}_{W^+(T), n_T}; R_T ) \\
\to  B(\widehat{\cV}_{W^+(T), n_T}; \varepsilon_T )/ S(\widehat{\cV}_{W^+(T), n_T}; \varepsilon_T ), 
 \end{split}
\end{align}
given in the previous subsection. By the assumption $\Psi(W, \xi, \s) \neq 0$, we can see that $\mathcal{M}_{\eta} (T) \neq \emptyset$. Using \cref{existence}, we obtain a translation invariant solution to \eqref{SWcyl} in a temporal gauge. Then one can use \cite[Lemma 9]{KM97} and obtain the conclusion. 
\qed

\section{Adjunction inequality for relative Bauer--Furuta invariant}\label{Adjunction2}
In this section, we prove an adjunction inequality for relative Bauer--Furuta invariant of 4-manifolds with homotopy L-space-bounday. 

\subsection{Seiberg--Witten Floer homotopy type and relative Bauer--Furuta invariant} 
In this subsection, we review Manolescu's construction of the Seiberg--Witten Floer homotopy type. For the details, see \cite{Man03}. 

Let $Y$ be a rational homology $3$-sphere equipped a $\Spinc$-structure $\s$ and $g$ a Riemann metric on $ Y$.
The spinor bundle with respect to $\s$ is denoted by $S$. 

The map $\rho : \Lambda^*_Y\otimes \C  \to \operatorname{End} (S) $ denotes the Clifford multiplication induced by $\s$. The notation $B_0$ denotes a fixed flat $\Spinc$-connection. 
Then the set of $\Spinc$ connections can be identified with $i\Om^1 (Y)$. 
The {\it configuration space} is defined by
\[
\Con_{k-\frac{1}{2}} (Y) := L^2_{k-\frac{1}{2}} (i\Lambda^1_Y) \oplus  L^2_{k-\frac{1}{2}} (S) ,
\]
here $L^2_{k-\frac{1}{2}}$ denotes the completion with respect to $L^2_{k-\frac{1}{2}}$-norm.

We have the {\it Chern-Simons-Dirac functional }
\begin{align}
\mathcal{L} : \Con_{k-\frac{1}{2}} (Y) \to \R 
\end{align}
given as
\[
\mathcal{L} ( b, \psi) = -\frac{1}{2} \int_Y b \wedge db + \frac{1}{2} \int_Y \langle \psi, D_{B_0+ b} \psi \rangle d \operatorname{vol} , 
\]
where $D_{B_0+ b}$ is the $\Spinc$-Dirac operator with respect to the $\Spinc$-connection $B_0+ b$.
The {\it gauge group} 
\[
\G_{k+\frac{1}{2}} (Y) :=\left\{  e^\xi  \middle| \xi \in  L^2_{k+\frac{1}{2}} (Y; i\R )\right\} 
\]
 acts on $\Con_{k-\frac{1}{2}} (Y)$ by
\[
u \cdot (b, \psi) := (b- u^{-1} d u, u \psi). 
\]

 Since the normalized gauge group 
\[
\G^0_{k+ \frac{1}{2} } (Y) :=\left\{  e^\xi  \in\G_{k+\frac{1}{2}} (Y)   \middle|  \int_Y \xi  d\operatorname{vol} =0 \right\}
\] 
freely acts on $\Con_{k-\frac{1}{2}} (Y)$, one can take a slice. The slice is given by 
\[
V_{k-\frac{1}{2}} (Y) :=\ker   \left(d^{*}: L^2_{k-\frac{1}{2}}(i\Lambda^1_Y) \to L^2_{k-\frac{3}{2}} (i\Lambda^0_Y) \right) \oplus L^2_{k-\frac{1}{2}} (S). 
\]
 The formal gradient field of the Chen-Simons Dirac functional with respect to a norm, introduced by Manolescu, is the sum 
 \[
 l+ c:  V_{k-\frac{1}{2}}(Y)\to V_{k-\frac{3}{2} } (Y),
 \]
  where 
  \[
  l(b, \psi) = (*db , D_{B_0} \psi)
  \]
   and 
   \[
   c(b, \psi) = (\pr_{\ker d^{*}}  \rho^{-1} ((\psi \psi^*)_0) , \rho (b) \psi- \xi (\psi) \psi). 
   \]
Here $\xi (\psi ) \in i \Om^0(Y)$ is determined by the conditions 
 \[
 d \xi (\psi) = ( 1- \pr_{\ker d^{*}} ) \circ \rho^{-1} ((\psi \psi^*)_0) \text{ and }  \ \int_Y \xi (\psi) =0. 
 \]
 Note that $l+c$ is $S^1$-equivariant, where the $S^1$-action is coming from 
 \[
 S^1 = \G_{k+\frac{1}{2}} (Y) /  \G^0_{k+\frac{1}{2}} (Y). 
 \]
 For a subset $I \subset \R$, a map $x= ( b, \psi) : I \to V_{k-\frac{1}{2}} (Y)$ is called a {\it Seiberg--Witten trajectory} if 
 \begin{align}\label{grad}
 \frac{\partial}{\partial t} x(t) = - (l+c) (x(t)). 
 \end{align}
 \begin{defn}
 A Seiberg--Witten trajectory $x = ( b, \psi) : I \to V_{k-\frac{1}{2}} (Y)$ is {\it finite type} if 
 \[
 \sup_{ t\in I} \|\psi(t) \|_{Y}  <\infty \text{ and }\sup_{ t\in I} |\mathcal{L}(x(t)) |<\infty. 
 \]
 \end{defn}
 We consider subspaces
 $V^\mu_\lambda(Y)$ defined as the direct sums of eigenspaces whose eigenvalues of $l$ are in $(\lambda, \mu]$ for $\lambda< 0 <\mu$ and denote $L^2$-projection from $V_{k-\frac{1}{2}} (Y)$ to $V^\mu_\lambda(Y) $ by $ p^\mu_\lambda$.
  Then the finite dimensional approximation of \eqref{grad} is given by 
  \begin{align}\label{flow}
  \frac{\partial}{\partial t} x(t) = - (l + p^\mu_\lambda c )(x(t)) , 
  \end{align}
  where $x$ is a map from $I \subset \R$  to $V^{\mu}_\lambda (Y)$. 
  Manolescu(\cite{Man03}) proved the following result: 
  \begin{thm}The following results hold. 
  \begin{itemize}
  \item There exists $R>0$ such that all finite type trajectories $x: \R \to V_{k-\frac{1}{2}}(Y)$ are contained in $\overset{\circ}{B}(R;V_{k-\frac{1}{2}}(Y) )$, where $\overset{\circ}{B}(R; V_{k-\frac{1}{2}}(Y))$ is the open ball with radius $R$ in $V_{k-\frac{1}{2}}(Y)$. 
  \item 
  For sufficiently large $ \mu$ and $-\lambda$ and the vector field 
  \[
  \beta (l + p^\mu_\lambda c )
  \]
   on $V^\mu_\lambda(Y)$, $\overset{\circ}{B}(2R;V^\mu_\lambda(Y)) $ is an isolating neighborhood, where $\beta$ is $S^1$-invariant bump function such that $\beta|_{\overset{\circ}{B}(3R)^c}=0$ and  $\beta|_{\overset{\circ}{B}(2R)}=1$.  
  \end{itemize}
  \end{thm}
  Then an $S^1$-equivariant Conley index $I^\mu_\lambda$ depending on $V^\mu_\lambda(Y) $, the flow \eqref{flow}, an isolating neighborhood $\overset{\circ}{B}(2R)$ and its isolated invariant set is defined.
   Then the {\it Seiberg--Witten Floer homotopy type} is defined by 
  \[
 SWF(Y, \s) :=  \Sigma^{-n(Y, \s, g) \C -V^0_\lambda}  I^\mu_\lambda, 
  \]
 as a stable homotopy type of a pointed $S^1$-space,  where $n(Y, \s, g)$ is given by 
  \[
  n(Y, \s, g) := \ind_\C^{APS} (D^+_A ) - \frac{c^2_1( \s_X )-\sigma(X)}{8} . 
  \]
   Here $(X, \s_X)$ is a compact $\Spinc$ bounding of $(Y, \s)$, the used Riemann metric of $X$ is product near the boundary, $\ind_\C^{APS} (D^+_A )$ is the Atiyah-Patodi-Singer index of the operator $D^+_A$ and a $\Spinc$ connection $A$ is a $\Spinc$ connection on $X$ which is an extension of $B_0$. 
  For the meaning of formal desuspensions, see \cite{Man03}.

  Next, we summarize the definition of the relative Bauer--Furuta invariant $BF(X, \s_X)$ following \cite{Khan15}, \cite{Man03} and \cite{Man07}. 
Let $X$ be a compact oriented Riemannian 4-manifold with $\partial X=Y$ is a rational homology 3-sphere.
Assume the collar neighborhood of $\partial X$ is isometric to the product.
Let $\s_X$ be a $\Spinc$ structure on $X$ and give $Y$ the $\Spinc$ structure $\s$ obtained by restricting $\s_X$ to $Y$.
We denote the spinor bundles of $\s_X$ by $S_X=S^+_X\oplus S^-_X$ and the spinor bundle of $\s$ by $S$. 
For simplicity, assume $b_1(X)=0$
\par
Let $\Omega^1_{CC}(X)$ be the space of 1-forms $a$ on $X$ in double Coulomb gauge.
The relative Bauer--Furuta invariant $BF(X, \s_X)$ arises as the finite-dimensional approximation of the Seiberg--Witten map
\begin{align}
\mathcal{F}^\lambda_X: L^2_{k}(i\Lambda^1_{X})_{CC}\oplus L^2_k(S^+_X)\to& L^2_{k-1}(i\Lambda^+_{X}\oplus S^-_X)\oplus V^\lambda_{-\infty} (Y)\\
(a, \phi)\mapsto& (d^+ a-\rho^{-1}(\phi\phi^*)_0, D^+_{A_0}\phi+\rho(a)\phi, p^\lambda_{-\infty}\circ  r(a, \phi))
\end{align}
for $\lambda \in \R$.
We will denote
\[
\cU_X=L^2_{k}(i\Lambda^1_{X})_{CC}\oplus L^2_k(S^+_X)\text{ and}\quad \cV_X=L^2_{k-1}(i\Lambda^+_{X}\oplus S^-_X). 
\]
We will also sometimes denote the map to the first two factors by $L_X+C_X$, 
where $L_X=d^++D^+_{A_0}+p^\lambda_{-\infty}r$ and $C_X$ is compact.
The finite-dimensional approximation goes as follows.
Pick an increasing sequence $\lambda_n\to \infty$ and an increasing sequence of finite-dimensional subspaces $\cV_{X, n} \subset \cV_X$ with $\pr_{\cV_{X, n}}\to 1$ pointwise.
Let 
\[
\cU_{X, n}=(L_X+p^{\lambda_n}_{-\infty}r)^{-1}(\cV_{X, n} \times V^{\lambda_n}_{-\lambda_n}) \subset \cU_X,
\]
 and 
\[
\cF_{X, n}:=    P_n \circ \cF^{\lambda_n}_X: \cU_{X, n}\to \cV_{X, n}\oplus V^{\lambda_n}_{-\lambda_n}, 
\]
where $P_n := \pr_{\cV_{X, n}} \times \pr_{V^{\lambda_n}_{-\lambda_n}} $.
Let 
\[
\wt{K}^1_{X, n}= (\cF_{X, n})^{-1}(\overline{B}(\cV_{X, n}; \varepsilon_n)\times  V^{\lambda_n}_{-\lambda_n}) \cap \overline{B}(\cU_{X, n}, R),
\]
\[
 \quad \wt{K}^2_{X, n}=(\cF_{X, n})^{-1}(\overline{B}(\cV_{X, n}; \varepsilon_n)\times  V^{\lambda_n}_{-\lambda_n})\cap S(\cU_{X, n}, R)
\]
\[
{K}^1_{X, n}=\pr_{V^{\lambda_n}_{-\lambda_n}}\circ \cF_{X, n}(\wt{K}^1_{X, n}), \quad K^2_{X, n}=\pr_{V^{\lambda_n}_{-\lambda_n}}\circ \cF_{X, n}(\wt{K}^2_{X, n})
\]
for some $R>0$.
One can find an $S^1$-equivariant index pair $(N_{X,n}, L_{X,n})$ which represents the Conley index for 
$V^{\lambda_n}_{-\lambda_n}$ in the form $N_{X, n}/L_{X, n}$ such that $K^1_{X, n}\subset N_{X, n}$ and $K^2_{X, n}\subset L_{X ,n  }$.
\par
Now, for a sufficiently large $n$, we have a map
\[
\mathcal{F}_{X, n}: \overline{B}(\cU_{X, n}, R) /S(\cU_{X, n}, R)\to  (\cV_{X, n}/(\overline{B}(\cV_{X, n}, \varepsilon)^c)) \wedge (N_{X, n}/L_{X, n}).
\]
This gives the relative Bauer--Furuta invariant $BF(X, \s_X)$ constructed by Khandhawit(\cite {Khan15}) and Manolescu(\cite{Man03}).
\par
We first state a non-vanishing result for the relative Bauer--Furuta invariant, which is a consequence of the non-vanishing result \cite[Corollary  4.3]{I19} of \eqref{iida} and gluing results \cite{Man07, KLS18II, IT20}. Before that recall a $\Spinc$ rational homology 3-sphere $(Y,\s)$ is a {\it (Floer) homotopy L-space} if 
\[
SWF(Y, \s) \cong (\C^\delta)^+
\]
for some $\delta  \in \Q$. It is known that if $Y$ is homotopy L-space for any $\Spinc$-structure, then $Y$ is an L-space 
\cite[Definition 5.1]{IT20}. 
For example, any rational homology 3-spheres with positive scalar curvature metric are homotopy L-spaces for any $\Spinc$ structures.
\begin{thm}\label{b=1}(\cref{b=1'})
Let $(W, \om)$ be a weak symplectic filling of a contact 3-manifold $(Y,\xi)$. Suppose $(Y, \s_\xi)$ is a homotopy L-space. 
We also consider a closed $\Spinc$ 4-manifold $(X, \s_X)$ with $b_1(X)=0$ such that the $S^1$-equivariant Bauer--Furuta invariant $BF(X, \s_X)$ is not $S^1$-null-homotopic. 
Then,  the relative Bauer--Furuta invariant 
\[
BF (W\# X, \s_{\om} \# \s_X):  (\R^{-b^+_2(W\# X)}\oplus \C^{\frac{c^2_1(\s_{\om} \# \s)-\sigma(W\# X)}{8}})^+\to SWF(Y, \s_\xi)
\] 
is equal to $BF(X, \s_X)$ and in particular does not vanish as an $S^1$-equivariant stable homotopy class.
 
\end{thm}
The proof of this result uses Seiberg--Witten Floer homotopy contact invariant introduced in \cite{IT20}.

\subsection{Proof of \cref{b=1} }
In order to prove non-vanishing result in \cref{b=1}, we need to introduce Seiberg--Witten Floer homotopy contact invariant
\begin{align}
    \Psi(Y, \xi):  S^0 \to SWF^{\frac{1}{2}-d_3(-Y; \xi) } (Y, \s_\xi) 
\end{align}
for a contact 3-manifold $(Y, \xi)$ with $b_1(X)=0$. 
For this invariant, the following lemma is proved: 
\begin{lem}\cite[Lemma 5.20]{IT20}\label{hom}
Let $Y$ be a rational homology $3$-sphere with a positive scalar curvature $g_Y$ and a fillable contact structure $\xi$ whose filling satisfies $b_3=0$.
Then 
$\Sigma^{(\frac{1}{2} - d_3 (-Y, \xi))  \R  } SWF(-Y, \s_{\xi})$ is stably homotopy equivalent to $S^0$ and  
\[
\Psi (Y, \xi ) : S^0 \to S^0 
\]
is a homotopy equivalence. (The map $\Psi (Y, \xi )$ is a generator of $\pi_0^S$.) Moreover, for a symplectic filling $(W, \om)$ of $(W, \xi)$, the non-equivariant relative Bauer--Furuta invariant of $(W, \s_{\om})$
\[
BF(W, \s_{\om}) : S^0 \to S^0
\]
 is also a homotopy equivalence.
\end{lem}
Note that, in the proof of \cite[Lemma 5.20]{IT20}, we only use the fact 
\[
SWF(Y, \s) = (\C^\delta)^+
\]
for a rational number $\delta \in \Q$.
We improve \cref{hom} as the following result: 
\begin{lem} \label{hom1}
Let $(Y, \xi)$ be a contact rational homology $3$-sphere with the condition 
\[
SWF(Y, \s_\xi) \cong (\C^\delta)^+
\]
for some $\delta \in \Q$ and $(W, \om)$ be a symplectic filling of $(Y, \xi)$ and $b_3(W)=0$.

Then, $BF(W, \s_{\om})$ is an $S^1$-equivariant homotopy equivalence. 
\end{lem}

\begin{proof}
Suppose $Y$ admits a fillable contact structure $\xi$ with a filling $(W,\om)$. We set $b_1(W)=0$. 
Then, \cref{hom} implies the Bauer--Furuta invariant
\[
BF(W, \s_{\om}) : S^0 \to S^0
\]
is a homotopy equivalence as non-equivariant stable homotopy classes. By \cite[Theorem 5.3]{IT20}, we can conclude that $b^+(W)=0$.

In general, the restriction of the relative Bauer--Furuta invariant $BF(W, \s_\om)$ of connected compact $\Spinc$ 4-manifold $(W, \s_\om)$ with $b_1(W)=0$ and $b_1(\partial W)=0$ to the $S^1$-invariant part
\[
BF(W, \s_\om)^{S^1}: (\R^m)^+  \to (\R^{m+ b^+(W)})^+
\]
is coming from linear inclusion \cite{Man03, Ma14}. 
In our situation, $BF(W, \s_{\om})^{S^1}$ is a homotopy equivalence since we have $b^+(W)=0$. 
It is proved in \cite{TD87} (used in \cite{B04}) that such an $S^1$-equivariant map is $S^1$-homotopic to the identity. 
So, we can conclude that $BF(W, \s_{\om})$ is an $S^1$-equivariant homotopy equivalence. 
\end{proof}
Now, we give a proof of \cref{b=1}. 
By \cref{hom1}, $BF(W, \s_{\om})$ is an $S^1$-equivariant homotopy equivalence. 
By applying the gluing formula for relative Bauer--Furuta invariants \cite{B04, Man07}, we can conclude that
\[
\Psi (W\# X, \s_{\om} \# \s_X) = \Psi(W, \s_{\om}) \circ  BF (X , \s_X) . 
\]
Here we actually used the following computations as $S^1$-equivariant maps: 
\begin{itemize}
    \item $BF(D^4) = \id$ (This is a special case of \cref{hom1}), 
    \item $BF(X^o, \s|_{X^o} ) =BF(X, \s)$, where $X^o$ is $X \setminus \operatorname{int} D^4$, and
    \item $BF(X_1 \# X_2, \s_1 \# \s_2 ) =BF(X_1^o, \s_1|_{X_1^o} ) \circ BF(X_2^o, \s_2|_{X_2^o} )$. 
\end{itemize}
This completes the proof of \cref{b=1}. 
Since $BF(W, \s_{\om})$ is an $S^1$-equivariant homotopy equivalence and $\Psi (X , \s)$ is not null homotopic as an $S^1$-equivariant map, $\Psi (W\# X, \s_{\om} \# \s)$ is not null homotopic as an $S^1$-equivariant map. 
This completes the proof of \cref{b=1}.

\subsection{Adjunction inequality for relative Bauer--Furuta invariant}
In the proof of \cref{cap}, we use adjunction inequality for relative Bauer--Furuta invariant proven in this subsection. 
As in the case of \cref{main2'}, we can see that (ii) follows from (i). 
The proof is essentially the same as in the proof of \cref{main2'}. 
\begin{thm} \label{main4'}
Let $(W, \s_W)$ be an oriented Spin$^c$ compact 4-manifold whose boundary is a $\Spinc$ rational homology 3-sphere $(Y,\s_Y)$ with $b_1(W)=0$. Suppose $(Y, \s_Y)$ is a homotopy L-space. 
If the relative Bauer--Furuta invariant satisfies $BF (W, \s) \neq 0$ as an $S^1$-equivariant stable homotopy class, then the following results hold: 
\begin{itemize}
\item[(i)]
There are no closed oriented 3-dimensional submanifolds $M \subset \operatorname{int}(W)$ satisfying the following three conditions.
\begin{enumerate}
\item
$M$ admits a Riemann metric with positive scalar curvature.
\item
$H^2(W, \partial W ; \R)\to H^2(M; \R)$ is non-zero.
\item $M$ separates $W$.
\end{enumerate}
\item[(ii)] Non-torsion homology class in $ H_2(W,\partial W; \Z)$ cannot be realized by an embedded 2-sphere whose self-intersection number is non-negative. 
\item[(iii)] For any connected, orientable, embedded, closed surface $\Sigma \subset W$ with $g(\Sigma)>0$ and $[\Sigma]\cdot [\Sigma] \geq 0$, we have 
\[
| \langle c_1 (\s_W) , [\Sigma] \rangle | + [\Sigma]\cdot [\Sigma]\leq 2 g(\Sigma) - 2. 
\]
\end{itemize}
\end{thm}
\begin{rem}
\cref{main4'}(ii) also follows from \cite[Corollary 1.9]{KLS18II}. 
\end{rem}


\begin{proof}[Proof of \cref{main4'}(i)]
We fix a Riemann metric on $W$ which is product near the boundary $\partial W=Y$.

In addition, suppose there is an embedded codimension-1 manifold  $M \subset \text{int}(W)$. 
We set $\overline{W}  :=W \setminus \overset{\circ}{\nu}(M)  $, where $\overset{\circ}{\nu} (M)$ is an even smaller open tubular neighborhood of $  M$. Then $\partial \overline{W}  = M \cup (-M)$.

Let $T$ be a non-negative real number and $W(T)$ be a Riemannian manifold obtained from $W$ by inserting a neck $[-T, T]\times M$ with a product metric: 
\[
W(T)=  [-T,  T]\times M \cup_{ M \cup (-M) } \overline{W}  
\]
such that the restriction of the metric of $W(T)$ on $\overline{W} $ coincides with a given metric. 
For each $T$, we have the $\eta$-perturbed Seiberg--Witten map 
\begin{align}
\begin{split}
\mathcal{F}^\lambda_{W(T), \eta} : \cU_{W(T)} \to \cV_{W(T)}  \oplus V^\lambda_{-\infty} (Y)\\
(a, \phi)\mapsto (d^+ a-\rho^{-1}(\phi\phi^*)_0 + \eta, D^+_{A_0}\phi+\rho(a)\phi, p^\lambda_{-\infty}\circ  r(a, \phi))
\end{split}
\end{align}
for a fixed closed two form $\eta$ on $W(T)$ such that $\supp \eta \subset \operatorname{int}W(T)$.

Pick an increasing sequence $\lambda_n\to \infty$ and an increasing sequence of finite-dimensional subspaces $\cV_{W(T), n} \subset \cV_{W(T)}$ with $\pr_{\cV_{W(T), n}}\to 1$ pointwise for each $T$. 
Let 
\[
\cU_{W(T), n}=(L_{W(T)}+p^{\lambda_n}_{-\infty}r)^{-1}(\cV_{W(T), n} \times V^{\lambda_n}_{-\lambda_n}) \subset \cU_{W(T)} ,
\]
 and 
\[
\cF_{W(T), n}:=    P_n \circ \cF^{\lambda_n}_{W(T)}: \cU_{W(T), n}\to \cV_{W(T), n}\oplus V^{\lambda_n}_{-\lambda_n}. 
\]
Also, as it is mentioned in the previous section, we can obtain 
\begin{align}
\mathcal{F}_{W(T), n} : \overline{B}(\cU_{W(T), n}, R) /S(\cU_{W(T), n}, R)\to (\cV_{W(T), n}/(\overline{B}(\cV_{W(T), n}, \varepsilon)^c)) \wedge (N_{T,n}/L_{T,n}) 
\end{align}
for some index pairs $(N_{T,n}, L_{T,n})$. 
\begin{lem}
Let $U=\R^m\oplus\C^n$, $V=\R^{m'}\oplus \C^{n'}$, $W=\R^{m''}\oplus \C^{n''}$ be finite dimensional vector spaces with $S^1$ action.
Here, the $S^1$-action on $\R$ is trivial and the $S^1=U(1)$ action on $\C$ is standard.
Let $(\mathcal{I}, y_0)$ be a pointed $S^1$-space equipped with an $S^1$-equivariant pointed homotopy equivalence $h: \mathcal{I} \to W^+$.
Suppose an $S^1$-equivariant pointed map 
\[
f: U^+\to V^+\wedge \mathcal{I}
\]
is given and
$f^{-1}(0, y)$ is empty for all $y \in I$.
Then, $f$ is $S^1$-null-homotopic.
\end{lem}
\begin{proof}
Define 
\[
f':=(id_{V^+}\wedge h)\circ f: U^+\to V^+ \wedge W^+.
\]
Then, we have the following commutative diagram up to based homotopy: 
\begin{align}
\begin{CD}\label{Tn1}
U^+@>>> V^+\wedge \mathcal{I} \\ 
@V{=}VV  @V{id_{V^+} \wedge h}VV \\ 
U^+@>>> V^+\wedge W^+, 
\end{CD}
\end{align}
It is sufficient to show that  $(f')^{-1}(0, 0)=\emptyset$, but this is obvious since the image of $f$ is contained in $V^+\wedge\{y_0\}$ by assumption and $h$ sends $y_0 \in \mathcal{I}$ to $+ \in W^+$.
\end{proof}

The following lemma is the key lemma to prove \cref{main4'}. 
\begin{lem}\label{non-emp}
For any $T$, there exists a solution $x_T$ to Seiberg--Witten equation on $W$ such that 
\[
\cL_\eta(x_T|_{\{-T\}\times M})-
\cL_\eta(x_T|_{\{T\}\times M})
\leq C. 
\]
\end{lem}
\begin{proof}
Apply the above lemma with $U=\cU_{W(T), n}$, $V=\cV_{W(T), n}$, 
\[
(\mathcal{I}, y_0)=(N_{T, n}/L_{T,n}, L_{T,n}/L_{T, n}),
\]
and
$f=\mathcal{F}_{W(T), n}$. 
Since $\mathcal{F}_{W(T), n}$ is not $S^1$ stable homotopy equivalent to the constant map and  we are assuming 
\[
SWF(Y, \s) \cong (\C^\delta)^+, 
\]
the above lemma implies that there is some element $y_{n, T} \in  N_{T}$ such that 
\[
x_{n,T} \in \mathcal{F}_{W(T), n}^{-1} (0, y_{n, T}) \neq \emptyset . 
\]
(Note that $f=\mathcal{F}_{W(T), n}$ itself is not necessarily surjective, unlike $\mathcal{F}'_{W(T), n}$.)
Using \cite{Khan15}, there is a subsequence $\{x_{n_j,T}\}$ of $\{x_{n,T}\}$ such that 
\[
x_{n_j,T} \to  x_T \quad \text{in }{L^2_{k} (W(T))} \text{ as } j\to \infty.
\]
The proof of boundedness of $\cL_\eta(x_T|_{\{-T\}\times M})-
\cL_\eta(x_T|_{\{T\}\times M})$ is essentially the same as in the proof of \cref{bounded}. 
\end{proof}

Then, the following proposition can be also proved using the same argument in \cref{existence}
\begin{prop} \label{existence1}
For each $T$, we have a solution $x_T$ to the $\eta$-perturbed Seiberg--Witten equation on $W(T)$ with 
\[
\cL_\eta(x_T|_{\{-T\}\times M})-
\cL_\eta(x_T|_{\{T\}\times M})
\leq C. 
\]
Consider the $\Spinc$ structure on the cylinder $\R\times M$ 
obtained as the pull back of $\s|_M$ by the projection $\R\times M\to M$ 

Then there exists a solution to the $\eta$-perturbed Seiberg--Witten equation
\begin{equation}\label{SWcyl}
\begin{split}
\frac{1}{2}F^+_{A^t}-\rho(\Phi\Phi^*)_0&=2i\hat{\eta}^+\\
D^+_A\Phi&=0
\end{split}
\end{equation}
 which is translation invariant and in a temporal gauge.
Here, $\hat{\eta}$ is the pull back of $\eta$ by the projection $\R\times M\to M$ and the word temporal gauge means that the $dt$ component of the $\Spinc$ connection vanishes.
\qed
\end{prop}
However, this contradicts  \cref{non-ex}. 
This completes the proof of \cref{b=1}(i).

\end{proof}

The proof of \cref{b=1}(iii) is essentially the same as that of \cref{main0}(ii). So we omit it.

\section{Proof of Applications}\label{Application}

In this section, we will first prove Theorem~\ref{main1} which obstructs codimension-0 embeddings of 4-manifolds in symplectic 4-manifolds with $b^+_2 \equiv 3 \operatorname{mod} 4$ and their connected sums. Part of this proof is inspired by the idea to find embedded 2-spheres given in \cite{Y19} by Yasui. Most of the applications that has been mentioned in the introduction followed from \cref{main1}.

\begin{thm}\label{main1}
Let $(W,\om)$ be a connected, weak symplectic filling of a given contact 3-manifold $(Y, \xi)$ with $b_3(W)=0$.
We consider the following two types of 4-manifolds:
\begin{itemize}
    \item $X_1$, $X_2$ and $X_3$ are closed sympectic 4-manifolds with $b_1(X_i)=0$ and $b^+_2(X_i) \equiv 3 \operatorname{mod} 4$, and
    \item $X$ is a closed negative definite 4-manifold  with $b_1(X)=0$.
\end{itemize}
If $W$ has a geometrically isolated 2-handle, then there is no orientation preserving embedding of $\overline{W}$ into $X\#X_1\#  X_2 \# X_3 $. In particular, there is no orientation preserving embedding from $\overline{W}$ to $X$, $X_1$, $X_1\# X_2  $, $X_1\#  X_2 \# X_3$,  $X_1\# X$ and $X_1\# X_2  \# X$. 
\end{thm}

\begin{proof}[Proof of Theorem~\ref{main1}]
  Let us assume that $\overline{W}$ can be embedded in $X$. In order to kill the homology of the boundary, we obtain a new 4-manifold $W_1$ from $W$ with boundary a rational homology sphere $Y_1$ by attaching Weinstein 2-handles on the boundary $Y$ along the generators of $b_1(Y)$. By Theorem~\ref{cob}, $W_1$ is a (weak) symplectic 4-manifold with convex boundary $(Y_1,\xi_1)$. Notice that $W\natural \overline{W}$ is embedded in $W_1\#X$, see Figure~\ref{connectedsum}.
 
\begin{figure}[htbp]
	\begin{center}
	
  \begin{overpic}[scale=0.6,   tics=20]{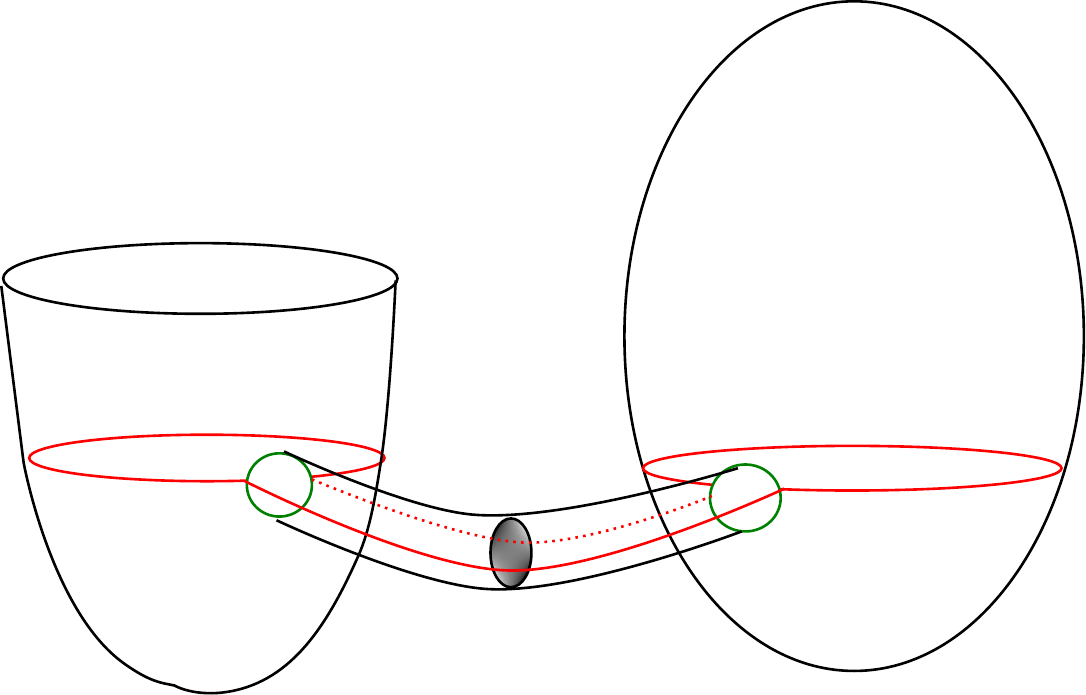}
  
  \put(35,-10){$W_1$}
  \put(140,-10){$X$}
  \put(35,13){\textcolor{red}{$W$}}
  \put(140,13){\textcolor{red}{$\overline{W}$}}

\end{overpic}

\caption{This is an analogous picture of $W\natural \overline{W}$ embedded in $W_1\# X$. }
\label{connectedsum}
\end{center}
\end{figure}
 
 Since $W$ has a geometrically isolated 2-handle, there exists a handle decomposition of $W$ with a 2-handle $h$ that doesn't intersect any 1-handles of $W$ and in particular of $W_1$. And moreover $h$ is non-trivial in $H_2(W_1, \partial W_1;\Z)/\operatorname{Tor} \cong H_2(W_1;\Z)/\operatorname{Tor}$. (Here we used $b_1(\partial W_1)=0$.) Now in $W\natural \overline{W}$, by sliding the 2-handle $h$ over $\overline{h}$ we obtain a new 2-handle $h'$, where $\overline{h}$ is a 2-handle of $\overline{W} $ corresponding to $h$.
 Since $h$ and $\overline{h}$ doesn't intersect any 1-handles of $W\natural \overline{W}$, the attaching sphere of $h'$ bounds a disk in the complement of the core of the 2-handle $h'$, see Figure~\ref{slide}. And thus after attaching the core of $h'$ we get a self-intersection $0$ 2-sphere $S$ embedded in $W_1 \# X$. Notice that the projection map from $H_2(S;\Z)$ to $H_2(W_1;\Z)/\operatorname{Tor}$ is non-trivial. And thus $S$ represent a non-trivial element in $H_2(W_1\# X;\Z)/\operatorname{Tor}.$ So by Theorem~\ref{main2}, $\Psi(W_1\# X, \xi_1, \s_1\#\s_X)=0$. But this contradicts the Theroem~\ref{main0}. Similar arguments work for $X_1\#X_2$ and $X_1\#X_2\#X_3$, since we have non-vanishing result \cref{main0} of Iida's invariant $\Psi$ for such manifolds. 
 
\begin{figure}[htbp]
	\begin{center}
	
  \begin{overpic}[scale=0.6,   tics=20]{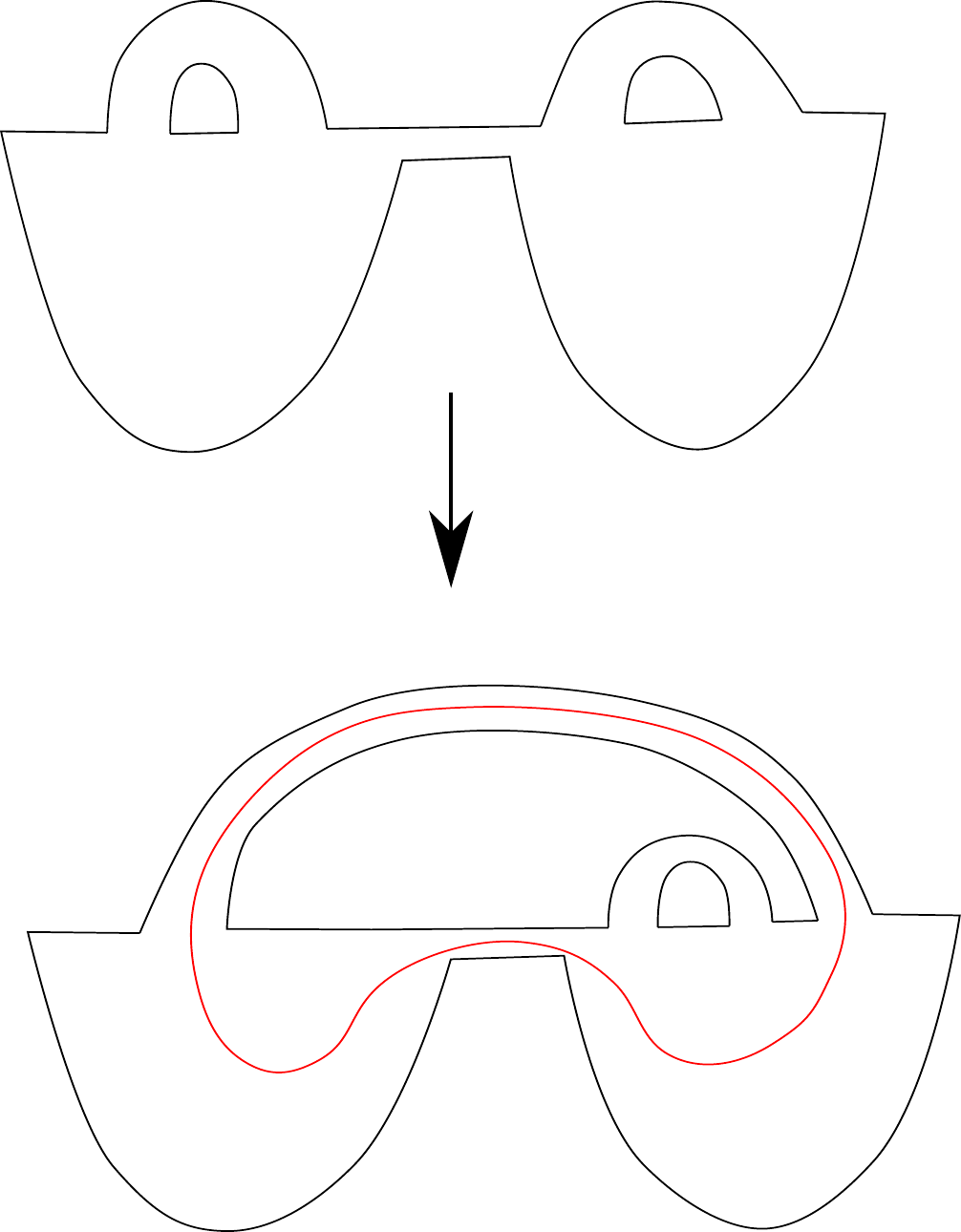}
  \put(35,225){$h$}
  \put(115,225){$\overline{h}$}
  \put(60,210){\textcolor{green}{$\longrightarrow$}}
  \put(85,130){Handle slide $h$ over $\overline{h}$}
  \put(80,100){$h'$}
  \put(155,40){\textcolor{red}{$S$}}
  \put(80,0){$W\natural \overline{W}$}
  \put(67,160){$W\natural \overline{W}$}

\end{overpic}

\caption{ After siding $h$ over $\overline{h}$ we get a self-intersection $0$ sphere $S$ in $W\natural\overline{W}$. }
\label{slide}
\end{center}
\end{figure}
\end{proof}

\begin{rem}\label{remark}
Notice that $b^+_2(X_i) \equiv 3 \operatorname{mod} 4$ in Theorem~\ref{main1}
is important. 
If $X_n(K)$ is an $n$-trace of a knot which is obtained by attaching a 2-handle on the 4-ball along a knot $K\subset S^3$ with framing $n$, then $X_n(K)$ can be embedded in $\C\mathbb P^2\# 2\overline{\C\mathbb P^2}$.  To see this let us assume the case that $n$ is an odd integer. Let $X$ be the result of attaching a 2-handle to $X_n(K)$ along a meridian of the knot $K$ with framing $0$. It is not difficult to show using Kirby calculus that $X$ has a handle decomposition  where two 2-handles are attached on the boundary of $B^4$ along two disjoint unknots with framing $1$ and $-1$. By capping off $X$ with a copy of $B^4$ we get $\C\mathbb P^2\# \overline{\C\mathbb P^2}.$ And in the case $n$ is an even integer, if we blow up the meridian of $K$ once, the framing of $K$ becomes odd. Now if we follow the previous process, we will get $\C\mathbb P^2\# 2\overline{\C\mathbb P^2}.$    
If $n< TB(K)$, then $X_n(K)$ admits a symplectic structure with convex boundary by Theorem~\ref{cob}. But $\overline{X_n(K)}= X_{-n}(\overline {K})$ can still be embedded in  $\C\mathbb P^2\# 2\overline{\C\mathbb P^2}$. This observation shows that the same statement as \cref{main1} does not hold for 4-manifolds with $b^+_2=1$. 

\end{rem}

Although, we imposed $b^+_2(X_i) \equiv 3 \operatorname{mod} 4$ in \cref{main1}, we can also treat symplectic 4-manifolds with $b^+_2(X_i) \equiv 1\operatorname{mod} 4$ if $Y$ has some additional condition. 
\begin{thm}\label{b=1appl}
Let $W$ be a connected, weak symplectic filling of a given contact rational homology 3-sphere $(Y, \xi)$ with $b_3(W)=0$. Suppose $Y$ has a positive scalar curvature metric and $W$ admits a geometrically isolated 2-handle. 
Then if $X$ is a closed symplectic 4-manifold with $b_1(X)=0$ and $b^+_2(X)\geq 2$, there is no orientation-preserving embedding of $\overline{W}$ into $X$.
\end{thm}
\begin{proof}[Proof of Theorem~\ref{b=1appl}]
The relative Bauer-Furuta invariant of $W\# X$ is non-trivial by \cref{b=1'}.
Now the proof of \cref{b=1appl} is the same as that of \cref{main1}, using Theorem~\ref{main4'} instead of \cref{main2'}. 
\end{proof}

Next we will prove \cref{cap}.
This proof uses relative Bauer--Furuta invariant but doesn't use Iida's invariant \eqref{iida}.
\begin{proof}[Proof of \cref{cap}]
Suppose, on the contrary, there exists such a positive definite symplectic cap $X$.
We can take a strong symplectic filling $W$ with $b_1=0$ of $(Y, \xi)$ by the assumption.
The closed manifold $W\cup X$ obtained by gluing $W$ and $X$ along the boundaries has the glued symplectic structure, which we denote by $\Omega$.

Let us denote $X$ with the opposite orientation by $\overline{X}$, which is negative definite and whose oriented boundary is $Y$ .
By assumption, we can take a $\Spinc$ structure $\s_{X}$ on $X$ such that
\[
\frac{-c^2_1( \s_X)+ b_2(X)}{8}= \delta . 
\]

We regard $\s_X$ also as a $\Spinc$ structure on $\overline{X}$.
By the assumption that $Y$ has a positive scalar curvature metric, the Seiberg--Witten Floer homotopy type of $(Y, \s_\xi)$ is given by
\[
SWF(Y, \s_\xi)=(\C^\delta)^+, 
\]
where $\delta=\delta(Y, \s_\xi)$ is the
Fr\o yshov invariant.
Now the relative Bauer--Furuta invariant of  $(\overline{X}, \s_X)$ is a morphism
\[
BF(\overline{X}, \s_X): (\R^0\oplus \C^{\frac{-c^2_1(\s_X)_X+b_2(X)}{8}})^+\to (\R^0\oplus \C^{\delta})^+.
\]

Note that the domain and the codomain have the same dimension here.




The $S^1$-invariant part of $BF(\overline{X}, \s_X)$ is induced by a linear isomorphism and thus homotopy equivarence, so  $BF(\overline{X}, \s_X)$ is $S^1$ equivariant homotopy equivalence by \cite[Lemma 3.8]{BF04}.
On the other hand, since $W\cup X$ is a closed symplectic 4-manifold with $b^+\geq 2$ and $b_1=0$ (We use the assumption $b^+(X)\geq 2$, $b_1(X), b_1(W)=0$ here.), its $S^1$-equivariant Bauer--Furuta invariant $BF(W\cup X, \s_\Omega)$ is non-trivial.

Combined with the gluing result of \cite{Man07, KLS18II},   
the relative Bauer--Furuta invariant of $(W\cup X)\# \overline{X}$ is given by
\[
BF((W\cup X)\# \overline{X}, \s_\Omega \# \s_X)=BF(W\cup X, \s_\Omega)\wedge BF(\overline{X}, \s_X)=BF(W\cup X, \s_X).
\]
Thus $BF((W\cup X)\# \overline{X}, \s_\Omega \# \s_X)$ is non-trivial.
However, as given in the proof of \cref{main1}, we can construct an embedded 2-sphere $S$ inside $(W\cup X)\# \overline{X}$ as a non-torsion element of $H_2((W\cup X)\# \overline{X}, \partial ((W\cup X)\# \overline{X}); \Z)$  with self intersection number $0$.  So, using \cref{main0}, one can conclude that 
\[
BF((W\cup X)\# \overline{X}, \s_\Omega \# \s_X)=0 . 
\]
This gives a contradiction. 
\end{proof}

Now we will prove \cref{cap1} which is giving some constraints on the topology of symplectic caps of $S^3$ and $\Sigma(2,3,5).$
\begin{proof}[Proof of \cref{cap1}]
We prove the case of $S^3$ and that of $\Sigma(2, 3, 5)$ respectively: 
\begin{itemize}
    \item 
It is sufficient to show that any compact oriented positive definite 4-manifold $X$  with $b_1(X)=0$ and $\partial X=-S^3$ admits a $\Spinc$ structure $\s_X$ satisfying 
\[
\frac{-c^2_1(\s_X)+b_2(X)}{8}=0. 
\]
By Donaldson's theorem A \cite{Do83}, we can conclude that the intersection form of $X$ is isomorphic to $ (+1)^{\oplus b_2(X)}$. 
We write by $\{v_i\}_{1 \leq i \leq b_2(X)}$ a basis of the intersection form of $X$ representing $(+1)^{\oplus b_2(X)}$.

Since $v_1 +\cdots + v_{b_2(X)}$ is a characterisic vector, there exist a $\Spinc$ structure $\s_X$ such that its first Chern class is equal to it.
Obviously $c^2_1(\s_X)=b_2(X)$ holds, so this is a desired $\Spinc$ structure.
\item 
It is sufficient to show that any compact oriented positive definite 4-manifold $X$  with $b_1(X)=0$ and $\partial X=-\Sigma(2,3,5)$, and having no 2-torsion on its homology admits a $\Spinc$ structure $\s_X$ satisfying 
\[
-c^2_1(\s_X)+b_2(X)=8. 
\]
By Scaduto's theorem \cite[Theorem 1.3]{Sca18}, we can conclude that the intersection form of $X$ is isomorphic to $E_8 \oplus  (+1)^{\oplus (b_2(X)-8)} $. We write by $\{v_i\}_{1 \leq i \leq b_2(X)}$ a basis of the intersection form of $X$ representing $E_8 \oplus  (+1)^{\oplus(b_2(X)-8)} $. 
Since $v_9 +\cdots + v_{b_2(X)}$ is a characterisic vector, there exist a $\Spinc$ structure $\s_X$ such that its first Chern class is equal to it.
Obviously $c^2_1(\s_X)=b_2(X)-8$ holds, so this is a desired $\Spinc$ structure.
\end{itemize}
\end{proof}

 Now we will prove Theorem~\ref{ben} which is a Bennequin type inequality for symplectic caps.
 
 \begin{proof}[Proof of Theorem~\ref{ben}]
 If $K$ bounds a disk in $X$ with self-intersection number $n>0$, then when consider $X$ as a symplectic cap of $S^3$, the boundary orientation get reversed. Thus by gluing a symplectic filling $(D^4, \omega_{std})$, we can get a closed symplectic 4-manifold where $X_n(\overline{K})$ is smoothly embedded. But notice that by assumption, $\overline{X_{n}(\overline{K})}= X_{-n}(K)$ which is Stein fillable. So the proof of Theorem~\ref{main1} and \cref{main2'}(ii) gives a contradiction.  
 \end{proof}
 \begin{rem}\label{additional}
 By the similar idea of the proof of Theorem~\ref{ben}, one can prove that: 
 Let $(X,\om)$ be a symplectic cap for $(S^3,\xi_{std})$ with $b_1(X)= 0$ and $b^+_2(X) \equiv 3 \operatorname{mod} 4$. If $K$ is a Legendrian knot in $(S^3,\xi_{std})$ with $tb(K)>0$, then $K$ does not bound a smoothly embedded surface $\Sigma$ in $X$ with $[\Sigma]\cdot[\Sigma]\geq 0$, $g(\Sigma)>0$ and 
 \[
 | \langle c_1 (\s_X) , [\Sigma] \rangle | + [\Sigma]\cdot [\Sigma]\leq 2 g(\Sigma) - 2, 
 \]
where $\s_X$ is the $\Spinc$ structure coming from $\om$ and $\langle c_1 (\s_X) , [\Sigma] \rangle$ is defined by taking a bound of $\overline{K}$ in $D^4$. We easily see $\langle c_1 (\s_X) , [\Sigma] \rangle$ does not depend on the choices of such a bounding. 
In the proof, we use \cref{main2}(iii). 
 \end{rem}
Now we will prove \cref{h-slice} which is giving an obstruction for a knot being H-slice in a 4-manifold with boundary $S^3.$
 \begin{proof}[Proof of \cref{h-slice}]
 Notice that if $K$ is H-slice in $X^{\circ}$ then $X_0(\overline{K})$, which has an geometrically isolated 2-handle, embedded in $X$. But follows from Theorem~\ref{cob} $X_0(K)$ has a symplectic structure with convex boundary since $TB(K)>0$. Thus it is contradicting Theorem~\ref{main1}. Same arguments work for $X_1\#X_2$ and $X_1\#X_2\#X_3$.
 \end{proof}
 
 \begin{rem}\label{fdim}
 In \cite{MR06}, it is proven that, for a $\Spinc$ 4-manifold with contact boundary, the generalized Bennequin inequality holds when the Kronheimer--Mrowka's invariant \cite{KM97} is not zero. 
 \cref{h-slice} can be seen as results for symplectic caps. Here we confirm that Kronheimer--Mrowka's invariant always vanishes for any symplectic cap with $(S^3, \xi_{std})$-boundary. 
 Recall that Kronheimer--Mrowka's invariant is defined to be zero unless the virtual dimension of the moduli space is zero.
 We show that for any symplectic cap $X$ with $(S^3, \xi_{std})$-boundary, the virtual dimension is odd. 

This follows from the following two facts by regarding $X =(X\cup_{S^3} D^4)\# D^4$:
\begin{enumerate}
\item
For any closed oriented 4-manifold $X$, the parity of the virtual dimension of the Seiberg--Witten moduli space is independent of the $\Spinc$ structure.
Moreover, they are even if there exists an almost complex structure  on $X$.
\item
Let $(X_1, \s_1)$ be a closed oriented $\Spinc$ 4-manifold and $(X_2, \xi, \s_2)$ be a $\Spinc$ 4-manifold with contact boundary  with $\s|_{\partial X_2}=\s_\xi$.
Then
\[
\langle e(S^+_{X_1\# X_2}, \Phi_0), [X_1\# X_2, \partial (X_1\# X_2)]\rangle = \langle e(S^+_{X_1}), [X_1]\rangle+\langle e(S^+_{X_1}, \Phi_0), [X_2, \partial X_2]\rangle+1, 
\]
where $\Phi_0$ is a non-vanishing section on $\partial (X_1\# X_2)=\partial X_2 $ constructed as in Section 2.

\end{enumerate}
We first prove (1). 
Recall that the virtural dimension is the sum of the index for $\Spinc$ Dirac operator and Atiyah--Hitchin--Singer (AHS) operator.
The former is even since $\Spinc$ Dirac operator is complex operator and AHS operator does not involve $\Spinc$ structure.
This implies the first part of (1).

The virtual dimension of Seiberg--Witten moduli space for  a closed $\Spinc$ 4-manifold $(X, \s)$ is given by
\[
\langle e(S^+), [X]\rangle=\frac{1}{4}(c_1(\s)^2-2\chi(X)-3\sigma(X))
\]
(See \cite[lemma 28.2.3]{KM07},for example).
If $\s$ comes from an almost complex structure, we can give a non-vanishing section of $S^+$
(See \cite[lemma 2.1]{KM97}, for example).
Thus $e(S^+)=0$. This implies the second claim of (1).
\par
We next prove (2). 
We denote 
$X^\circ_i$ the manifold obtained from $X_i$ by removing a small 4-ball from the interior and regard
\[
X_1\# X_2=X^\circ_1\cup_{S^3}X^\circ_2.
\]

We also consider 
\[
S^4=D^4_+\cup_{S^3} D^4_-.
\]
equipped with the unique $\Spinc$ structure.
We fix a non-vanishing section of the spinor bundle of the unique $\Spinc$ structure on $S^3$.
We identify this $S^3$  with $\partial X^\circ_1$ and $\partial D^4_+$ at the same time, so we obtained a non-vanishing section of the positive spinor bundle on
\[
\partial D^4_+ \amalg \partial X^\circ_1 \amalg \partial X_2 \subset S^4\amalg  (X^\circ_1\cup_{S^3}X^\circ_2).
\]
We can extend  this section to $X_1\# X_2$ so that the zero set is isolated.
This section can be also regarded as a section on
\[
(X^\circ_1\cup_{S^3} D^4_-)\amalg (D^4_+ \cup_{S^3} X^\circ_2)=X_1 \amalg X_2.
\]
in the obvious way.
The number of zero points counted with sign gives the desired equality.
Here, we use the fact that the contribution from  $S^4$ is $-1$.
Alternatively, we can deduce the same equality from excision principle of index.
 \end{rem}

  Now we will prove \cref{infinite} by constructing infinitely many topologically H-slice but not smoothly H-slice knots.
 \begin{proof}[Proof of \cref{infinite}]
 Let $K$ be a knot in $S^3$ with $TB(K)>0$. It is easy to see that the untwisted positive Whitehead double $Wh_0^+(K)$ of $K $ has $TB(Wh_0^+(K))>0.$ Since $Wh_0^+(K)$ is topologically slice in $B^4$, it is topologically H-slice in $X$. But it is not smoothly H-slice in $X$ by Theorem~\ref{h-slice}. Now, we take a concrete family of knots as torus knots $\{T_{2, 2^n-1}\}_{n\in \Z_{>0}} $. In \cite{HK12}, Hedden-Kirk proved that $\{Wh_0^+ (T_{2, 2^n-1}) \}_{n\in \Z_{>0}}$ are linearly independent in the knot concordance group. 
  \end{proof}
 
 \begin{rem}
 There are many preceding studies of linear independence of infinitely many topologically slice knots in the knot concordance group. Let $\mathcal{T}$ be the subgroup of the knot concordance group generated by topologically slice knots. 
 Endo \cite{E95} proved that $\mathcal{T}$ contains $\Z^\infty$ (a certain class of preztel knots) as a subgroup
 by using Furuta's result \cite{Fu90}. There are many developments of studies of $\Z^\infty$-subgroups in $\mathcal{T}$: 
 In Yang-Mills gauge theory side, there are several related studies of $\Z^\infty$-subgroups in $\mathcal{T}$ \cite{HK12, PJ17, NST19}.
 In Heegaard Floer theory, several concordance invariants were constructed: the tau-invariant $\tau : \mathcal{C} \to \Z $ \cite{OS03}, the nu$^+$ invariant $\nu^+: \mathcal{C} \to \Z_{\geq 0}   $ \cite{HW16}, and the upsilon invariant $\Upsilon : \mathcal{C} \to \operatorname{PL}([0,2], \R) $ \cite{OSS17}. 
Moreover, there are several studies finding $\Z^\infty$-subgroup or summands in $\mathcal{T}$ in Heegaard Floer theory \cite{OSS17, HK18, Hom19, ASA20}.

 \end{rem}
 Now we will prove \cref{top h} which focuses on knots in $K3^{\circ}$.
 \begin{proof}[Proof of \cref{top h}]
 For a knot $K$ with $TB(K)>0$ if we prove that $Wh_0^+(K)$ is smoothly slice in $K3$ then by using Theorem~\ref{infinite} we can prove that it is topologically H-slice and smoothly slice but not smoothly H-slice. Since the unknotting number of $Wh_0^+(K)$ is 1, by \cite[Corollary 2.8]{manolescu}, it is smoothly slice in $K3$.
 Now, we take a concrete family of knots as torus knots $\{T_{2, 2^n-1}\}_{n\in \Z_{>0}} $. 
 Again, in \cite{HK12}, Hedden-Kirk proved that $\{Wh_0^+ (T_{2, 2^n-1}) \}$ are linearly independent in the knot concordance group.
 \end{proof}
 
 Now we will use our techniques to obstruct knots in $S^3$ from being rationally slice.

\begin{proof}[Proof of Theorem~\ref{slice}]
Let $K$ be a slice knot in a rational homology ball $W$ with boundary $S^3$ and $TB(K)>0$. We obtained $X$, a closed rational homology 4-sphere, by capping-off $W$ with a 4-ball. If $K$ is slice in $W$ then $X_n(\overline{K})$ can be embedded in $X$ for some $n\in \Z$. We will first show that $n=0$. If not, then boundary of $X_n(\overline{K})$ is a rational homology sphere. By Mayer--Vietoris formulation, $b_2(X)\geq b_2(X_n(\overline{K}))= 1$. But $X$ is a rational homology 4-sphere, so contradiction. Since $TB(K)>0$, $X_0(K)$ has a symplectic structure by Theorem~\ref{cob}. Attach a Weinstein 2-handle on $X_0(K)$ along the meridian of $K$ to obtain a new symplectic 4-manifold $X_1$ with convex boundary $(Y_1,\xi_1).$ Note that $X_1$ has an isolated 2-handle and $b_3(X_1)=0$ by the construction. 
But similar to the proof of Theorem~\ref{main1} we can prove that there exists a self intersection 0 sphere which represent a non-trivial element of $H_2(X_1\#X;\Z)/\operatorname{Tor}$. So $\Psi(X_1\#X,\xi_1)=0$ by Theorem~\ref{main2}. But this contradicts Theorem~\ref{main0}.
\end{proof}

\begin{proof}[Proof of \cref{neg def inf}]
Again, we set 
\[
K_n :=Wh_0^+ (T_{2, 2^n-1}). 
\]
We saw $TB(K_n)> 0$ and $\{K_n\}$ are linearly independent by using the result in \cite{HK12}. Now the result follows from \cref{slice}. 
\end{proof}

 Finally we will use our invariant to detect exotic structures on compact 4-manifolds with boundary.
\begin{proof}[Proof of Theorem~\ref{exotic}]

Notice that if a 3-manifold $Y$ bounds an exotic pair of 4-manifolds $X$ and $X'$ then by reversing orientation, $\overline{Y}$ bounds an exotic pair $\overline{X}$ and $\overline{X'}$. So, without loss of generality, we assumed that $(W,\om)$ is a weak-symplectic filling of $(Y,\xi)$ with $b_3(W)=0$ by Corollary~\ref{main0}. If $X= W\# K3\#\overline{\C \mathbb{P}^2}$ then $\Psi(X,\xi,\s)\neq 0$. So by the result of Freedman, \cite{freedman} $K3\#\overline{\C \mathbb{P}^2}$ is homeomorphic to $3\C\mathbb{P}^2\#20\overline{\C\mathbb{P}^2}$ and thus $X$ is homeomorphic to $X'= W\#3\C\mathbb{P}^2\#20\overline{\C\mathbb{P}^2}$. Now notice that there exists a smooth self-intersection $0$ 2-sphere representing the element $\alpha + \overline{\alpha} \in H_2(\C\mathbb{P}^2\#\overline{\C\mathbb{P}^2};\Z)\subset H_2(X';\Z)/\operatorname{Tor}.$ So by Theorem~\ref{main2}, $\Psi(X',\xi)=0$. Thus $X$ and $X'$ are not diffeomorphic. 
\end{proof}

\bibliographystyle{plain}
\bibliography{tex}

\end{document}